\documentclass[11pt]{amsart}
\usepackage[centertags]{amsmath}
\usepackage{amsfonts}
\usepackage{amssymb}
\usepackage{amsthm}
\usepackage{newlfont}
\usepackage{amscd}
\usepackage{mathrsfs}
\usepackage[all,cmtip]{xy}
\usepackage{tikz-cd}
\tikzcdset{arrow style=Latin Modern, row sep/normal=1.35em, column sep/normal=1.8em, cramped}
\usepackage{tikz}
\usepackage[pagebackref=false,colorlinks]{hyperref}

\definecolor{mycolor}{HTML}{F7F8E0}
\definecolor{myorange}{RGB}{245,156,74}
\definecolor{cadetgrey}{rgb}{0.57, 0.64, 0.69}
\definecolor{calpolypomonagreen}{rgb}{0.12, 0.3, 0.17}

\hypersetup{pdffitwindow=true,linkcolor=calpolypomonagreen,citecolor=calpolypomonagreen,urlcolor=calpolypomonagreen}
\usepackage{cite}
\usepackage{fancyhdr}
\usepackage{stmaryrd}

\usepackage[OT2,OT1]{fontenc}
\newcommand\cyr{%
\renewcommand\rmdefault{wncyr}%
\renewcommand\sfdefault{wncyss}%
\renewcommand\encodingdefault{OT2}%
\normalfont
\selectfont}
\DeclareTextFontCommand{\textcyr}{\cyr}

\usepackage[toc, page] {appendix}

\numberwithin{equation}{section}

\newtheorem{thm}{Theorem}[section]

\newtheorem{cor}[thm]{Corollary}
\newtheorem{lem}[thm]{Lemma}
\newtheorem{prop}[thm]{Proposition}

\newtheorem{conj}[thm]{Conjecture}

\theoremstyle{definition}
\newtheorem{defn}[thm]{Definition}

\newtheorem{rem}[thm]{Remark}

\newtheorem{ques}[thm]{Question}

\newcommand{\KS}{\mathbf{KS}}

\newcommand{\ks}{\boldsymbol{\kappa}}

\newcommand{\sha}{\textrm{{\cyr SH}}}

\begin{document}
\title[$\sha$ and $\ks^{\mathrm{Hg}} \neq 0$]{On the arithmetic of Tate--Shafarevich groups via Kolyvagin's conjecture}
\author[C.-H. Kim]{Chan-Ho Kim}
\address{
Department of Mathematics and Institute of Pure and Applied Mathematics,
Jeonbuk National University,
567 Baekje-daero, Deokjin-gu, Jeonju, Jeollabuk-do 54896, Republic of Korea
}
\email{chanho.math@gmail.com}
\thanks{Chan-Ho Kim was partially supported 
by a KIAS Individual Grant (SP054103) via the Center for Mathematical Challenges at Korea Institute for Advanced Study,
by the National Research Foundation of Korea(NRF) grant funded by the Korea government(MSIT, ME) (No. 2018R1C1B6007009, 2019R1A6A1A11051177, RS-2024-00339824, RS-2025-16067678, RS-2026-25593558), 
by Global-Learning \& Academic research institution for Master’s$\cdot$Ph.D. Students, and Postdocs (LAMP) Program of the National Research Foundation of Korea (NRF) funded by the Ministry of Education (No. RS-2024-00443714),
by the International Centre for Theoretical Sciences (ICTS) for the program Automorphic Forms and the Bloch--Kato Conjecture (code: ICTS/afbk2025/05), and
by the POSCO Science Fellowship of POSCO TJ Park Foundation.
}
\date{\today}
\begin{abstract}
We investigate how Heegner points detect the finiteness of the $p$-primary part of Tate--Shafarevich groups of semi-stable elliptic curves over the rationals of \emph{arbitrary} rank with any good reduction prime $p \geq 5$ with irreducible mod $p$ representation and partially recover \c{C}iperiani--Wiles' theorem on solvable points on genus one curves. 
As a key ingredient of the proof of these results, we also give a concise proof of Kolyvagin's conjecture for semi-stable elliptic curves with supersingular reduction with a relevant choice of an imaginary quadratic field.
Our approach is independent of the cyclotomic Iwasawa main conjecture for elliptic curves with supersingular reduction.

%
\end{abstract}
\maketitle
\setcounter{tocdepth}{1}


\section{Introduction}
Let $E$ be an elliptic curve over $\mathbb{Q}$ without complex multiplication and $p \geq 5$ be a good reduction prime
for $E$.
We say that \emph{$E[p]$ has large image} if the mod $p$ representation $\overline{\rho} : \mathrm{Gal}(\overline{\mathbb{Q}}/\mathbb{Q}) \to \mathrm{Aut}_{\mathbb{F}_p}(E[p])$ is surjective.
We assume this large image assumption throughout this article.

The aim of this article is to present three results on the following topics:
\begin{enumerate}
\item[Thm.\:$ $\ref{thm:main-finiteness}.] Certain families of Heegner points still detect the finiteness of the $p$-primary part of the Tate--Shafarevich group $\sha(E/\mathbb{Q})[p^\infty]$ when the analytic rank of $E$ is greater than 1 (a partial generalization of the theorem of Gross--Zagier and Kolyvagin).
\item[Thm.\:$ $\ref{thm:main}.] Heegner point Kolyvagin systems are non-trivial for semi-stable elliptic curves with supersingular reduction (Kolyvagin's conjecture for the supersingular setting).
\item[Cor.\:$ $\ref{cor:ciperiani-wiles}.] Under mild assumptions, a smooth projective curve over $\mathbb{Q}$ of genus one admits a point in a solvable extension (a partial result on the theorem of \c{C}iperiani--Wiles with a refinement on the solvable extensions).
\end{enumerate}

\subsection{The finiteness of Tate--Shafarevich groups}
\subsubsection{}
Let $K$ be an imaginary quadratic field such that the (generalized) Heegner hypothesis holds, $p$ splits in $K$, and the quadratic twist of $E$ by $K$ has analytic rank $\leq 1$.
See $\S$\ref{subsubseec:choice-K} for the precise conditions on $K$.

The ground-breaking work of Gross--Zagier and Kolyvagin and the recent advances on its $p$-converse roughly establish  the equivalence of the following statements:
\begin{enumerate}
\item The vanishing order of the $L$-function of $E$ over $K$ at $s=1$ is one. 
\item The Heegner point over $K$ is non-torsion. 
\item The Mordell--Weil rank is one and the Tate--Shafarevich group is finite.
\item The $p^\infty$-Selmer corank of $E$ over $K$ is one (and the $p$-primary part of the Tate--Shafarevich group of $E$ over $K$ is finite) for one $p$.
\end{enumerate}
We investigate how this equivalence statement extends to elliptic curves of arbitrary rank.
Our result says that, even when the analytic rank is greater than one, the finiteness of $\sha(E/\mathbb{Q})[p^\infty]$ is still detected by the arithmetic of Heegner points (but in a more complicated way).
As far as we know, there is almost no theoretical result towards the finiteness of (the $p$-primary part of) Tate--Shafarevich groups of elliptic curves when the analytic rank is $> 1$ \cite{coates-sha}.
\subsubsection{}
Let 
$$\partial = \partial^{(\mathrm{ord}(\ks^{\mathrm{Hg}}))} (\ks^{\mathrm{Hg}}) \geq 0$$
 be a constant determined purely by the collection of Heegner points over ring class fields
 whose precise definition is given in (\ref{eqn:numerical-invariants}).
Let $k \geq 1$ be an integer such that $p^{k+\partial-1} \mathrm{Sel}(K, E[p^\infty])_{/\mathrm{div}} = 0$ where $\mathrm{Sel}(K, E[p^\infty])_{/\mathrm{div}}$ is the quotient of $\mathrm{Sel}(K, E[p^\infty])$ by its maximal divisible subgroup. The minimal choice of such an integer $k$ can be computed from modular symbols \cite{kim-structure-selmer}.
Let $n_k \in \mathcal{N}_{k+\partial}$
where $\mathcal{N}_{k+\partial}$ is the set of square-free products of mod $p^{k+\partial}$ Kolyvagin primes for $(E, p, K)$ reviewed in \S\ref{subsubsec:construction-kolyvagin-classes}.
Let $K[n_k]$ be the ring class field of $K$ of conductor $n_k$. In particular, $K[1]$ is the Hilbert class field of $K$.
Let
$$P_{n_k} \in E(K[n_k])$$ be the derived Heegner point over $K[n_k]$ reviewed in \S\ref{subsubsec:derived-heegner-points}, and denote by 
$$\kappa^{\mathrm{Hg}}_{n_k} \in p^{\partial}\mathrm{Sel}(K, E[p^{k+\partial}]) \subseteq \mathrm{Sel}(K, E[p^{k}])$$
the corresponding cohomology class in the Heegner point Kolyvagin system $\ks^{\mathrm{Hg}}$ which is defined by the image of the mod $p^{k+\partial}$ reduction of $P_{n_k}$.
Kolyvagin's conjecture \cite[Conj. A]{kolyvagin-selmer} claims that $\kappa^{\mathrm{Hg}}_{n_k}$ is non-trivial for some $n_k$ with $k \geq 1$, and it is now known for a large class of triple $(E, p, K)$ (Theorem \ref{thm:main-ordinary}) and we also prove some new cases in this article (Theorem \ref{thm:main}).
Assume that
$$r^+_p := \mathrm{cork}_{\mathbb{Z}_p} \mathrm{Sel}(\mathbb{Q}, E[p^\infty]) \geq 2 .$$
Under his conjecture, Kolyvagin determined $r^+_p$ purely in terms of the collection of Heegner points over ring class fields (Theorem \ref{thm:kolyvagin-vanishing-order}) and proved that there exist $r^+_p$ derived Heegner points
\begin{equation} \label{eqn:generators-derived-heegner-points}
 P_{n_{k, 1}}, \cdots ,  P_{n_{k, r^+_p}} 
\end{equation}
such that their image in $\mathrm{H}^1(K, E[p^{k+\partial}])$ generates $p^{\partial}\mathrm{Sel}(K, E[p^{k+\partial}]) \simeq (\mathbb{Z}/p^k\mathbb{Z})^{\oplus r^+_p}$
where  $n_{k, i} \in \mathcal{N}_{k+\partial}$ and $\nu(n_{k, i}) = r^+_p - 1$ for every $1 \leq i \leq r^+_p$
 (Theorem \ref{thm:kolyvagin-construction-theorem-original}).
Here, $\nu(n)$ means the number of prime divisors of $n$.
The existence of such points is a part of his construction of Selmer groups (Theorem \ref{thm:kolyvagin-construction-theorem-original}) and has been curiously overlooked for decades.
Since the prime divisors of $n_{k,i}$'s are chosen by the Chebotarev density argument, the derived Heegner points in (\ref{eqn:generators-derived-heegner-points}) are not necessarily unique.

Our first main result shows that the finiteness of $\sha(E/\mathbb{Q})[p^\infty]$ is 
detected by the precise locations of the derived Heegner points mentioned above in the Mordell--Weil groups over the ring class fields and vice versa.
\begin{thm} \label{thm:main-finiteness}
Let $E$ be an elliptic curve over $\mathbb{Q}$ without complex multiplication and $p \geq 5$ be a good reduction prime
such that $E[p]$ has large image. If $E$ has supersingular reduction at $p$, we assume that $E$ is semi-stable.
Suppose that 
$$r^+_p = \mathrm{cork}_{\mathbb{Z}_p}\mathrm{Sel}(\mathbb{Q}, E[p^\infty]) \geq 2.$$
Then the following statements are equivalent:
\begin{enumerate}
\item $\sha(E/\mathbb{Q})[p^\infty]$ is finite.
\item For every integer $k \geq 1$ with $p^{k+\partial-1} \mathrm{Sel}(K, E[p^\infty])_{/\mathrm{div}} = 0$ and every $n_k \in \mathcal{N}_{k+\partial}$ with $\nu(n_k) = r^+_p - 1$, the derived Heegner point $P_{n_k} \in E(K[n_k])$ satisfies
 $$P_{n_k} \in p^{ \partial }\left( E(\mathbb{Q}) +p^kE(K[n_k]) \right).$$ 
\item  Fix an integer $k \geq 1$ with $p^{k+\partial-1} \mathrm{Sel}(K, E[p^\infty])_{/\mathrm{div}} = 0$.  the derived Heegner points $P_{n_{k, i}} \in E(K[n_{k, i}])$ in (\ref{eqn:generators-derived-heegner-points}) with $1 \leq i \leq r^+_p$ satisfies
 $$P_{n_{k, i}} \in p^{ \partial }\left( E(\mathbb{Q}) +p^kE(K[n_{k, i}]) \right)$$
\end{enumerate}
\end{thm}
\begin{proof}
See $\S$\ref{subsec:proof-theorem-main-finiteness}.
\end{proof}
It is remarkable that the finiteness question of Tate--Shafarevich groups can be translated to the statement purely on Heegner points even for elliptic curves of arbitrary rank.
When $\mathrm{cork}_{\mathbb{Z}_p}\mathrm{Sel}(\mathbb{Q}, E[p^\infty]) \leq 1$, the finiteness of $\sha(E/\mathbb{Q})[p^\infty]$ (and even $\sha(E/\mathbb{Q})$) follows from the strong $p$-converse to the theorem of Gross--Zagier and Kolyvagin. See \cite{wei-zhang-mazur-tate, burungale-skinner-tian-wan} for example.
\subsubsection{}
We also develop the \emph{local} analogue of Theorem \ref{thm:main-finiteness} based on the local-global principle on the linear dependence in Mordell--Weil groups in Theorem \ref{thm:main-finiteness-local-criteria} and \S\ref{subsec:linear-dependence}. The local analogue yields that the finiteness of $\sha(E/\mathbb{Q})[p^\infty]$ can be verified by \emph{finitely} many computations in elliptic curves over \emph{finite} fields.
This computation becomes effective provided that the reductions of Mordell--Weil groups over ring class fields at various primes can be computed effectively.

\subsubsection{}
We record a particularly interesting application of Theorem \ref{thm:main-finiteness} and its interpretations.
Under the finiteness of $\sha(E/\mathbb{Q})[p^\infty]$, we can easily characterize
$E(\mathbb{Q})/p^{k} E(\mathbb{Q}) \simeq p^\partial E(\mathbb{Q})/p^{k+\partial} E(\mathbb{Q})$ for every $k \geq 1$.
\begin{cor} \label{cor:approximate-kummer-image}
Under the assumptions of Theorem \ref{thm:main-finiteness}, if $\sha(E/\mathbb{Q})[p^\infty]$ is finite, then 
the image of the derived Heegner points in (\ref{eqn:generators-derived-heegner-points}) modulo $p^{k+\partial}$ coincides with the image of $p^\partial E(\mathbb{Q})/ p^{k+\partial} E(\mathbb{Q})$ in $p^\partial\mathrm{Sel}(\mathbb{Q}, E[p^{k+\partial}])$ for every $k \geq 1$.
\end{cor}
\begin{proof}
It follows from Theorem \ref{thm:main-finiteness}(3).
The phrase ``the image of the derived Heegner points in (\ref{eqn:generators-derived-heegner-points}) modulo $p^{k+\partial}$" can be replaced by
``the image of the derived Heegner points $P_{n_k}$ in Theorem \ref{thm:main-finiteness}(2) modulo $p^{k+\partial}$" if we use Theorem \ref{thm:main-finiteness}(2) instead.
\end{proof}
\begin{rem}
Corollary \ref{cor:approximate-kummer-image} has two different interpretations: a $p$-adic approximation of Mordell--Weil groups and the refined Mazur--Tate type conjecture for Heegner points.
\begin{enumerate}
\item Let $n''_k$ be the square-free product of all the prime divisors of $n_{k, i}$'s in (\ref{eqn:generators-derived-heegner-points}) for $1 \leq i \leq r^+_p$
and $L$ be the compositum of all the $K[n''_k]$'s for $k \geq 1$.
Then Corollary \ref{cor:approximate-kummer-image} is equivalent to the equality of subsets  of $E(K[n''_{k}])$
 $$p^{\partial}E(\mathbb{Q}) + p^{k+\partial}E(K[n''_{k}]) = \left\langle P_{n_{k, i}} : 1 \leq i \leq r^+_p \right\rangle  + p^{k+\partial}E(K[n''_{k}])  $$ for every $k \geq 1$.
Furthermore, it implies the equality of subsets of $E(L)$
 $$p^{\partial}E(\mathbb{Q}) + E(L)_{\mathrm{tors}} = \bigcap_{k \geq 1} \left( \left\langle P_{n_{k, i}} : 1 \leq i \leq r^+_p \right\rangle + p^{k+\partial}E(L) \right)   .$$
Hence,  the collection of such $P_{n_{k, i}}$'s (or $P_{n_k}$'s) yields a $p$-adic approximation of the Mordell--Weil group $E(\mathbb{Q})$ (modulo torsions) under the finiteness of $\sha(E/\mathbb{Q})[p^\infty]$ by Corollary \ref{cor:approximate-kummer-image}. It is well-known that $E(L)_{\mathrm{tors}}$ is finite (e.g. \cite{nekovar-schappacher, jetchev-lauter-stein}) and see \cite[Prop. 7]{moon-mordell-weil} for the structure of $E(L)$.
\item 
Corollary \ref{cor:approximate-kummer-image} also provides a more explicit and refined version of some results of Darmon and Bertolini--Darmon \cite[Conj. 2.3(2)]{darmon-refined-bsd} and \cite[the proof of Prop. 4.18]{bertolini-darmon-derived-heights-1994}  on the refined Mazur--Tate type conjecture for Heegner points.
Corollary \ref{cor:approximate-kummer-image} also generalizes \cite[Props. 5.9 and 5.10]{darmon-refined-bsd} in the following sense. Under various technical assumptions, Darmon proved that if a Darmon--Kolyvagin cohomology class is non-trivial mod $p$ at an appropriate spot, then $\sha(E/K)[p] = 0$ and the mod $p$ cohomology class actually lies in the image of the Mordell--Weil group $E(K)$ in $\mathrm{Sel}(K, E[p])$.
Here, this mod $p$ non-vanishing at an ``appropriate spot" is actually equivalent to $\mathrm{dim}_{\mathbb{F}_p}\mathrm{Sel}(K, E[p]) = \mathrm{dim}_{\mathbb{F}_p} E(K)/pE(K)$ from the viewpoint of Kolyvagin's triangulation of mod $p$ Selmer groups \cite{kolyvagin-structure-sha, kolyvagin-selmer, wei-zhang-mazur-tate}, so the consequence $\sha(E/K)[p] = 0$ is natural.
In this sense, Corollary \ref{cor:approximate-kummer-image} generalizes Darmon's mod $p$ result to the $p^\infty$-version.
\end{enumerate}
\end{rem}


\subsection{Kolyvagin's conjecture at good (supersingular) primes}
Combining with the work of Burungale--Castella--Skinner, Burungale--Castella--Grossi--Skinner, and the author \cite{burungale-castella-skinner-gl2,burungale-castella-grossi-skinner-indivisibility, kim-gross-zagier},
it is now known that Kolyvagin's conjecture is true for non-CM elliptic curves with good ordinary reduction at $p \geq 5$ such that $E[p]$ has large image as follows.
\begin{thm}[Burungale--Castella--Grossi--Skinner, Kim] \label{thm:main-ordinary}
Let $E$ be a non-CM elliptic curve over $\mathbb{Q}$ and $p \geq 5$ be a good ordinary reduction prime for $E$ such that $E[p]$ has large image.
Let $K$ be an imaginary quadratic field $K$ satisfying the classical Heegner hypothesis, $p$ splits in $K$, and  the quadratic twist of $E$ by $K$ has analytic rank $\leq 1$.
Then Kolyvagin's conjecture is true, i.e. the Heegner point Kolyvagin system $\ks^{\mathrm{Hg}}$ associated to $(E, p, K)$ is non-trivial.
\end{thm}
We give a succinct proof of Kolyvagin's conjecture for semi-stable elliptic curves with good supersingular reduction at $p \geq 5$.
 It is the key ingredient of Theorem \ref{thm:main-finiteness} for the supersingular reduction primes.
\begin{thm} \label{thm:main}
Let $E$ be a semi-stable elliptic curve over $\mathbb{Q}$ and $p \geq 5$ be a good supersingular reduction prime for $E$.
Let $K$ be an imaginary quadratic field satisfying the conditions given in $\S$\ref{subsubseec:choice-K} below.
Then Kolyvagin's conjecture is true, i.e. $\ks^{\mathrm{Hg}}$ is non-trivial.
\end{thm}
\begin{proof}
See $\S$\ref{sec:kolyvagin-conjecture}.
\end{proof}
Why do we care of Kolyvagin's conjecture? 
Under the non-triviality of $\ks^{\mathrm{Hg}}$, Kolyvagin presented the following striking applications without assuming any low analytic or Selmer rank condition \cite{kolyvagin-selmer}:
\begin{itemize}
\item[(str)]  The module structure of $\mathrm{Sel}(\mathbb{Q}, E[p^\infty])$ is determined in terms of $\ks^{\mathrm{Hg}}$ (Theorems \ref{thm:kolyvagin-vanishing-order} and \ref{thm:structure-kolyvagin}).
\item[(constr)]  $\mathrm{Sel}(\mathbb{Q}, E[p^\infty])$ is contained in the submodule generated by $\ks^{\mathrm{Hg}}$ in $\mathrm{H}^1(K, E[p^\infty])$ (Theorem \ref{thm:construction-rough}).
\end{itemize}
In particular, (constr) is one of the curious and unique features of Heegner points since there is no counterpart of (constr) for Kato's Euler systems. See \cite{mazur-rubin-book,kurihara-iwasawa-2012, kim-structure-selmer} for the analogue of (str) for the case of Kato's Euler systems.
Also, the strong rank one $p$-converse to the theorem of Gross--Zagier and Kolyvagin \cite[Thm. A]{castella-wan-perrin-riou-ss} follows easily from (str).

Combining Theorems \ref{thm:main-ordinary} and \ref{thm:main}, we have the following statement.
\begin{cor} \label{cor:main-all-but}
Let $E$ be a semi-stable elliptic curve over $\mathbb{Q}$.
Then the applications (str) and (constr) follow for every good reduction prime $p \geq 5$ with irreducible mod $p$ representation.
\end{cor}
It should be noted that there is no assumption on the tame ramification on the Galois representation in Theorems \ref{thm:main-ordinary} and \ref{thm:main}, and it is essential to have the ``global" applications of Corollary \ref{cor:main-all-but} including Corollary \ref{cor:ciperiani-wiles} on the theorem of \c{C}iperiani--Wiles below.
All but finitely many primes (e.g. every good reduction prime $p \geq 11$ \cite{mazur-rational-isogenies}) satisfy the condition in Corollary \ref{cor:main-all-but}.


\subsection{The existence of solvable points \`{a} la \c{C}iperiani--Wiles}
We partially recover the theorem of \c{C}iperiani--Wiles on the existence of solvable points on genus one curves with semi-stable Jacobian \cite{ciperiani-wiles}.
\begin{thm} \label{thm:ciperiani-wiles}
Let $E$ be a semi-stable elliptic curve over $\mathbb{Q}$ of conductor $N$.
Every element of $\sha(E/\mathbb{Q})[p^\infty]$ with $(N,p)=1$ splits in a finite ring class extension of a certain imaginary quadratic field.
In particular, every element of $\sha(E/\mathbb{Q})$ whose order is prime to $N$ splits in a finite solvable extension of $\mathbb{Q}$.
\end{thm}
\begin{proof}
For the mod $p$ representation of semi-stable elliptic curves, the large image assumption is equivalent to the irreducibility assumption \cite{edixhoven-serre-conjecture}.
By Corollary \ref{cor:main-all-but}, we are able to apply (constr) to every good reduction prime $p \geq 5$ with irreducible mod $p$ representation, the conclusion follows.
\end{proof}
The solvable extension used in Theorem \ref{thm:ciperiani-wiles} can be chosen to be unramified at $p$ and any given finite set of primes \cite[Rem. 19]{wei-zhang-mazur-tate}.
In this sense, this choice is more controlled than the solvable extension used in \cite{ciperiani-wiles} when the analytic rank is greater than one.
This refined control of the solvable extension plays an important role for Theorem \ref{thm:main-finiteness}.
For instance, no result like Theorem \ref{thm:main-finiteness} is observed in \cite{ciperiani-wiles} since their patching argument (cf. \cite{wiles,taylor-wiles}) constructs certain generalized Selmer groups which are larger than classical Selmer groups when the analytic rank is greater than one.

\begin{cor} \label{cor:ciperiani-wiles}
Let $C$ be a smooth projective curve of genus one over $\mathbb{Q}$ such that $C(\mathbb{Q}_v) \neq \emptyset$ for every place $v \leq \infty$.
If its Jacobian $J = \mathrm{Jac}(C)$ is a semi-stable elliptic curve of (square-free) conductor $N$ and it satisfies $\sha(J/\mathbb{Q})[N] = 0$, then $C$ admits a point in a solvable extension of $\mathbb{Q}$.
\end{cor}

\subsection{The idea of proof and the relations with other work}
In order to obtain Theorem \ref{thm:main-finiteness}, we carefully re-examine a part of Kolyvagin's Heegner point construction of Selmer groups. Then we classify which Kolyvagin cohomology classes at least conjecturally generate the Kummer image of $E(\mathbb{Q})$ by using Kolyvagin's structure theorem of Selmer groups.
Although Kolyvagin's construction theorem is used in proving his structure theorem, we revisit the construction theorem through the lens of the structure theorem to extract more refined information.

Our approach towards Kolyvagin's conjecture is based on Mazur--Rubin's approximation argument on the equivalence between the non-triviality of the specialization of the $\Lambda$-adic Kolyvagin system at a height one prime and the corresponding main conjecture localized at the same prime. See \cite{mazur-rubin-book}, \cite{kim-gross-zagier} for the good ordinary case, and \cite{kim-structure-selmer} for Kato's Euler systems.
A similar but different approach can also be found in \cite{burungale-castella-grossi-skinner-indivisibility, castella-sano} for the good ordinary case. See also \cite{lei-zhao-kolyvagin} for its non-ordinary variant.

After establishing the equivalence for the case of supersingular reduction, we apply the result of Castella--Wan on the supersingular analogue of Perrin-Riou's Heegner point main conjecture \cite{perrin-riou-heegner, castella-wan-perrin-riou-ss}, which is based on Heegner point Kolyvagin systems and Eisenstein congruences on $\mathrm{GU}(3,1)$ over $\mathbb{Q}$ \cite{castella-liu-wan}.

Regarding Perrin-Riou's Heegner point main conjecture and Kolyvagin's conjecture for the case of supersingular reduction,  there are several approaches including the works of Sweeting \cite{sweeting-kolyvagin}, Burungale--B\"{u}y\"{u}kboduk--Lei \cite{burungale-buyukboduk-lei-2}, Bertolini--Longo--Venerucci \cite{bertolini-longo-venerucci}, Longo--Pati--Vigni \cite{longo-pati-vigni}, and Da Ronche \cite{daronche}.
All these works basically adapt the approach of W. Zhang \cite{wei-zhang-mazur-tate} via Bertolini--Darmon's level raising congruence argument \cite{bertolini-darmon-imc-2005} and the cyclotomic main conjecture for elliptic curves with supersingular reduction \cite{kobayashi-thesis, kato-euler-systems,skinner-urban, wan_hilbert, burungale-skinner-tian-wan, burungale-castella-skinner-gl2}.

In these approaches, the Heegner point main conjecture is usually regarded an output of Kolyvagin's conjecture (as studied in \cite{burungale-castella-kim}), and certain hypotheses on the tame ramification are needed due to the arithmetic of Shimura curves.
Our approach reveals the precise relation between the Heegner point main conjecture and Kolyvagin's conjecture, is independent of the validity of the cyclotomic main conjecture for elliptic curves, and does not require any hypothesis on the tame ramification. As its consequence, our approach is more direct and simpler than others.

\subsection{Working hypotheses} \label{subsec:working-hypotheses}
We keep the following hypotheses throughout this article.
\subsubsection{}
Let $E$ be a semi-stable elliptic curve over $\mathbb{Q}$ of conductor $N$ and $p \geq 5$ be a good supersingular reduction prime for $E$, so $N$ is square-free and $a_p(E) = 0$.
In this case, it is well-known that the mod $p$ representation $\overline{\rho} : G_{\mathbb{Q}} = \mathrm{Gal}(\overline{\mathbb{Q}}/\mathbb{Q}) \to \mathrm{Aut}_{\mathbb{F}_p}(E[p]) \simeq  \mathrm{GL}_2(\mathbb{F}_p)$ is surjective \cite[Prop. 2.1]{edixhoven-serre-conjecture}.
By the rank one $p$-converse to the theorem of Gross--Zagier and Kolyvagin studied in \cite{castella-wan-perrin-riou-ss}, Theorem \ref{thm:main} is well-known when the Selmer rank is $\leq 1$. Thus, we assume that
$$\mathrm{cork}_{\mathbb{Z}_p}\mathrm{Sel}(\mathbb{Q}, E[p^\infty]) \geq 2.$$
\subsubsection{} \label{subsubseec:choice-K}
Given a pair $(E, p)$, we are able to choose an imaginary quadratic field $K$ of odd discriminant $-D_K < -4$
such that
$$N = N^+ \cdot N^-$$
where 
 a prime divisor of $N^+$ splits or is ramified in $K/\mathbb{Q}$, and
 a prime divisor of $N^-$ is inert in $K/\mathbb{Q}$.
We further assume that $K$ satisfies the following conditions:
\begin{itemize}
\item $\nu(N^-)$ is even.
\item When $N^- =1$, at least one prime divisor of $N$ is ramified in $K/\mathbb{Q}$.
\item 2 splits in $K/\mathbb{Q}$.
\item $p$ splits in $K/\mathbb{Q}$.
\item The quadratic twist $E^K$ by $K$ has analytic rank $\leq 1$.
\end{itemize}
Our choice of $K$ follows that in the proof of \cite[Thm. 6.11]{castella-wan-perrin-riou-ss}, and the existence of such an $K$ follows from the non-vanishing result \cite[Thm. B]{friedberg-hoffstein}.
Under these assumptions, the vanishing order of $L(E/K,s)$ at $s=1$ is odd, i.e. the root number $w(E/K)$  of $E$ over $K$ is $-1$.
Since a prime divisor of $N^+$ can be ramified in $K/\mathbb{Q}$, the case of mock Heegner points is essentially allowed.
\begin{rem}
When we treat the $p$-ordinary case, we only require the following conditions: 
\begin{itemize}
\item Every prime divisor of $N^+$ splits in $K/\mathbb{Q}$, so $N^- =1$.
\item $p$ splits in $K/\mathbb{Q}$.
\item The quadratic twist $E^K$ by $K$ has analytic rank $\leq 1$.
\end{itemize}
\end{rem}
\subsection{The organization of this article}
In \S\ref{sec:review}, we review anticyclotomic Selmer groups, Heegner points, and the Heegner point main conjecture for the supersingular setting as the preliminary material following \cite{castella-wan-perrin-riou-ss} closely.
In \S\ref{sec:kolyvagin-conjecture}, we prove Kolyvagin's conjecture for semi-stable elliptic curves with supersingular reduction.
In \S\ref{sec:applications}, we review Kolyvagin's Selmer structure and construction theorems and prove Theorem \ref{thm:main-finiteness}. Also, we discuss the local analogue of Theorem \ref{thm:main-finiteness}.
\section{Review on Selmer groups and Heegner points} \label{sec:review}
In order to fix the conventions and the notation, we review the notion of anticyclotomic Selmer groups of elliptic curves with supersingular reduction, the formulation and status of the corresponding Heegner point main conjecture, and the construction of Heegner point Kolyvagin systems for the supersingular setting.
We follow \cite[$\S$4 and Appendix A]{castella-wan-perrin-riou-ss} closely and all the details can be found therein. No new result is actually presented in this section.
\subsection{Signed Selmer groups}
\subsubsection{Preparations}
Let  $T =  \varprojlim_{m} E[p^m]$ be the $p$-adic Tate module of $E$, 
$V = T \otimes_{\mathbb{Z}_p} \mathbb{Q}_p$, and $W = V/T \simeq E[p^\infty]$.
Let $K_{\infty}$ be the anticyclotomic $\mathbb{Z}_p$-extension of $K$ and $\Lambda = \mathbb{Z}_p\llbracket \mathrm{Gal}(K_\infty/K) \rrbracket \simeq \mathbb{Z}_p\llbracket X \rrbracket$ the anticyclotomic Iwasawa algebra.
Denote by $K_m$ the cyclic subextension of $K$ of degree $p^m$ in $K_\infty$. 
By using Shapiro's lemma, we have isomorphism 
$\mathrm{H}^1_{\mathrm{Iw}}(K, T) := \varprojlim_m \mathrm{H}^1(K_m, T) \simeq \mathrm{H}^1(K, T \otimes \Lambda)$.

Let $\mathfrak{P} \neq p\Lambda$ be a height one prime of $\Lambda$.
Let $S_{\mathfrak{P}}$ be the integral closure of $\Lambda / \mathfrak{P}$, $\mathfrak{m}_{\mathfrak{P}} $ be its maximal ideal, and $\Psi_{\mathfrak{P}} = \mathrm{Frac} ( S_{\mathfrak{P}} )$.
Let
$T_{\mathfrak{P}} = ( T \otimes \Lambda ) \otimes_{\Lambda, \pi_{\mathfrak{P}}} S_{\mathfrak{P}}$
where 
$\pi_{\mathfrak{P}} : \Lambda \twoheadrightarrow \Lambda / \mathfrak{P} \hookrightarrow S_{\mathfrak{P}}$ is the specialization map at $\mathfrak{P}$.
We similarly define $V_{\mathfrak{P}} = T_{\mathfrak{P}} \otimes \Psi_{\mathfrak{P}}$ and $W_{\mathfrak{P}} = V_{\mathfrak{P}} / T_{\mathfrak{P}}$.
For a Galois module $M$, write 
$M^* = \mathrm{Hom}(M, \mu_{p^\infty})$
Then we have
$(T_{\mathfrak{P}})^*[\mathfrak{m}^k_{\mathfrak{P}}] = (T_{\mathfrak{P}} / \mathfrak{m}^k_{\mathfrak{P}} )^* $.
For a co-finitely generated $\Lambda$-module $M$, denote by $M^\vee$ the Pontryagin dual of $M$.

\subsubsection{Signed Selmer structures}
We recall the notion of signed Selmer structures on $\mathrm{H}^1(K, T \otimes \Lambda)$ and its specializations, and their dual sides.

Following \cite[Def. 4.6 and (4.7)]{castella-wan-perrin-riou-ss}, 
the signed $\Lambda$-adic Selmer structure $\mathcal{F}^{\pm}_{\Lambda}$ on $\mathrm{H}^1(K, T \otimes \Lambda)$ is defined by
\begin{itemize}
\item
$\mathrm{H}^1_{\mathcal{F}^{\pm}_{\Lambda}}(K_v, T \otimes \Lambda) = \mathrm{H}^1_{\pm}(K_v, T \otimes \Lambda)$ when $v$ divides $p$, and
\item $\mathrm{H}^1_{\mathcal{F}^{\pm}_{\Lambda}}(K_v, T \otimes \Lambda) = \mathrm{H}^1(K_v, T \otimes \Lambda)$ when $v$ does not divide $p$.
\end{itemize}
Although we do not reproduce the construction of $\mathrm{H}^1_{\pm}(K_v, T \otimes \Lambda)$ in this article, we make some remarks with precise reference.
The splitting of $p$ in $K/\mathbb{Q}$ plays an essential role in the construction of $\mathrm{H}^1_{\pm}(K_v, T \otimes \Lambda)$ because $\mathrm{H}^1_{\pm}(K_v, T \otimes \Lambda)$ is built upon $\mathrm{H}^1_{\pm}(\mathbb{Q}_p, T \otimes \Lambda)$ \cite[Def. 4.6]{castella-wan-perrin-riou-ss} and $\mathrm{H}^1_{\pm}(\mathbb{Q}_p, T \otimes \Lambda)$ is defined by the specialization of the two-variable signed local condition at $p$ studied in \cite[$\S$3]{castella-wan-perrin-riou-ss}.
We write
$\widehat{\mathrm{Sel}}_{\pm}(K_\infty, T) = \mathrm{H}^1_{\mathcal{F}^{\pm}_{\Lambda}}(K, T \otimes \Lambda) $.
Then the dual signed $\Lambda$-adic Selmer structure $\mathcal{F}^{\pm, *}_{\Lambda}$ on
$\mathrm{H}^1(K, (T \otimes \Lambda)^*)$
is defined by the orthogonal complements of the Selmer structure $\mathcal{F}^{\pm}_{\Lambda}$ via local Tate duality.
We write
$$\mathrm{Sel}^{\pm}( K_\infty, E[p^\infty]) = \mathrm{H}^1_{\mathcal{F}^{\pm, *}_{\Lambda}}(K, (T \otimes \Lambda)^*) .$$
Using $T \otimes \Lambda \to T \otimes \Lambda/ \mathfrak{P} \to T_{\mathfrak{P}} $,
the signed $\Lambda$-adic Selmer structure $\mathcal{F}^{\pm}_{\Lambda}$ induces 
the Selmer structures on $\mathrm{H}^1(K, T \otimes \Lambda/ \mathfrak{P})$
and $\mathrm{H}^1(K, T_{\mathfrak{P}})$ by propagation \cite[Ex. 1.1.2]{mazur-rubin-book}.

We denote both Selmer structures by $\mathcal{F}^{\pm}_{\mathfrak{P}}$ and define
\[
\xymatrix{
\mathrm{Sel}_{\pm}(K, T_{\mathfrak{P}})  = \mathrm{H}^1_{\mathcal{F}^{\pm}_{\mathfrak{P}}}(K, T_{\mathfrak{P}}), &
\mathrm{Sel}_{\pm}(K, T \otimes \Lambda/ \mathfrak{P})  = \mathrm{H}^1_{\mathcal{F}^{\pm}_{\mathfrak{P}}}(K, T \otimes \Lambda/ \mathfrak{P}).
}
\]
Similarly, using $(T_{\mathfrak{P}})^* \to (T \otimes \Lambda/ \mathfrak{P})^* = (T \otimes \Lambda)^*[\mathfrak{P}] \to (T \otimes \Lambda)^*$,
the dual signed $\Lambda$-adic Selmer structure $\mathcal{F}^{\pm,*}_{\Lambda}$ induces 
the Selmer structures on $\mathrm{H}^1(K, (T \otimes \Lambda/ \mathfrak{P})^*)$
and $\mathrm{H}^1(K, (T_{\mathfrak{P}})^*)$ by propagation again.
As before, we denote both (dual) Selmer structures by $\mathcal{F}^{\pm,*}_{\mathfrak{P}}$, and define
\[
\xymatrix{
\mathrm{Sel}^{\pm}(K, (T_{\mathfrak{P}})^*) = \mathrm{H}^1_{\mathcal{F}^{\pm,*}_{\mathfrak{P}}}(K, (T_{\mathfrak{P}})^*), &
\mathrm{Sel}^{\pm}(K, (T \otimes \Lambda/ \mathfrak{P})^*) = \mathrm{H}^1_{\mathcal{F}^{\pm,*}_{\mathfrak{P}}}(K, (T \otimes \Lambda/ \mathfrak{P})^*).
}
\]
It is known that the $\pm$-local conditions at the primes dividing $p$ are self-dual with respect to the local Tate pairing as explained in the proof of \cite[Thm. A.5]{castella-wan-perrin-riou-ss} which is based on \cite[Prop. 4.11]{bdkim-supersingular-parity}.
Therefore, the signed Selmer structures are self-dual, so they satisfy Howard's standard hypotheses for Heegner point Kolyvagin systems; in particular, \cite[H.4), p. 1446]{howard-kolyvagin}.

\subsubsection{Specializations}
We recall a comparison result and can identify the specialization of the plus Selmer group at the augmentation ideal with the classical Selmer group.
\begin{prop} \label{prop:specializations-Selmer}
There exists a finite set $\Sigma_{\Lambda}$ of height one primes of $\Lambda$ with $p\Lambda \in \Sigma_{\Lambda}$
such that for every $\mathfrak{P} \not\in  \Sigma_{\Lambda}$
the composite of natural maps
\[
\xymatrix@R=0em{
\dfrac{ \widehat{\mathrm{Sel}}_{\pm}(K_{\infty}, T) }{ \mathfrak{P} \widehat{\mathrm{Sel}}_{\pm}(K_{\infty}, T) }
 \ar[r] & \mathrm{Sel}_{\pm}(K, T \otimes \Lambda/ \mathfrak{P}) \ar[r]^-{(*)} & \mathrm{Sel}_{\pm}(K, T_{\mathfrak{P}})  , \\
 \mathrm{Sel}^{\pm}(K, (T_{\mathfrak{P}})^*) \ar[r]^-{(**)} &
 \mathrm{Sel}^{\pm}(K, (T \otimes \Lambda/ \mathfrak{P})^*) \ar[r] &
 \mathrm{Sel}^{\pm}(K_\infty, E[p^\infty])[\mathfrak{P}]
}
\]
have finite kernel and cokernel of order bounded by a constant depending only on $[S_{\mathfrak{P}} : \Lambda/\mathfrak{P}]$.
When $\mathfrak{P}$ is the augmentation ideal $X\Lambda$, the maps $(*)$ and $(**)$ are isomorphisms and
we have identifications $\mathrm{Sel}_{+}(K, T) = \mathrm{Sel}(K, T)$ and $\mathrm{Sel}(K, E[p^\infty]) =  \mathrm{Sel}^{+}(K, T^*)$.
\end{prop}
\begin{proof}
See \cite[Lem. 6.5]{castella-wan-perrin-riou-ss}, and \cite[(6.16)]{castella-wan-perrin-riou-ss} for the
$\mathfrak{P} = X\Lambda$ case.
\end{proof}

\subsubsection{Transverse local conditions}
For a prime $v$ of $K$ with residue characteristic not equal to $p$, the transverse local condition at $v$ is defined by
$\mathrm{H}^1_{\mathrm{tr}}(K_v, M) := \mathrm{ker} \left( \mathrm{H}^1(K_v, M) \to \mathrm{H}^1(L, M) \right)$
where $M$ is a $\mathrm{Gal}(\overline{K}_v/K_v)$-module and $L$ is a maximal totally tamely ramified abelian $p$-extension of $K_v$.
The  transverse local condition is also orthogonal to each other under the local Tate pairing \cite[Prop. 1.1.9]{howard-kolyvagin}.
The $n$-tranverse Selmer structure means that the local condition at the prime not dividing $n$ is the usual one and the local condition at the prime dividing $n$ is the transverse local one.

\subsection{Heegner points}
For more detailed constructions of Euler systems of Heegner points on modular and Shimura curves and the associated Kolyvagin systems, we refer to \cite{kolyvagin-euler-systems, gross-kolyvagin, ye-tian-thesis, howard-kolyvagin, howard-gl2-type, nekovar-euler-systems}.
\subsubsection{Heegner points on Shimura curves} \label{subsubsec:review-heegner-points}
Let $X_{N^+, N^-}$ be the Shimura curve over $\mathbb{Q}$ associated to the quaternion algebra over $\mathbb{Q}$ of discriminant $N^-$ and an Eichler order of level $N^+$, and $J_{N^+, N^-}$ be the Jacobian of $X_{N^+, N^-}$.
Fix a geometric Jacquet--Langlands parametrization
$\pi_{N^+, N^-} : J_{N^+, N^-}  \to E $.
For every positive integer $s$, $\mathcal{O}_s = \mathbb{Z} + s \mathcal{O}_K$ be the order of $K$ of conductor $s$
and denote by $K[s]$ the ring class field of $K$ of conductor s.
Then the Artin map induces an isomorphism $\mathrm{Gal}(K[s]/K) \simeq \mathrm{Pic}(\mathcal{O}_s)$.

Following \cite[Prop. 1.2.1]{howard-gl2-type}, there exists the collection of CM points on Shimura curve $X_{N^+, N^-}$ defined over ring class fields
$$x_s \in X_{N^+, N^-}(K[s])$$
satisfying the Euler system relation.
Sine $E[p]$ is irreducible, there exists a prime $q$ not dividing $Nq$ such that $a_q(E) - q -1$ is a $p$-adic unit.
The image of $x_s$ under the map 
\[
\xymatrix@R=0em@C=1em{
X_{N^+, N^-} \ar[r]^-{\iota_{N^+, N^-}} &  J_{N^+, N^-} \ar[r]^-{\pi_{N^+, N^-}} & E(K[s]) \ar[r]^-{\otimes (a_q(E) - q -1)^{-1}} & E(K[s]) \otimes \mathbb{Z}_p \\
x_s \ar@{|->}[r] & \iota_{N^+, N^-}(x_s) := (T_q - q - 1)x_s \ar@{|->}[rr] &  &
\pi_{N^+, N^-}(\iota_{N^+, N^-}(x_s)) \otimes (a_q(E) - q -1)^{-1}
}
\]
is denoted by $y_s \in E(K[n]) \otimes \mathbb{Z}_p$, and it is independent of the choice of $q$ where $\iota_{N^+, N^-}$ is the Abel--Jacobi map.

%
%
We write $s = np^m$ with positive integer $n$ prime to $p$ and integer $m \geq 0$, so $y_{np^m} \in E(K[np^m]) \otimes \mathbb{Z}_p.$ is the Heegner point of conductor $np^m$.
Then the \textbf{Heegner cohomology class over $K[np^m]$}
$$z[np^m] \in \mathrm{H}^1(K[np^m], T)$$
 is defined as the image of $y_{np^m}$ under the Kummer map
$E(K[np^m]) \widehat{\otimes} \mathbb{Z}_p \to \mathrm{H}^1(K[np^m], T) $.
Let $\mathrm{cores}_{m+1, m}$ denote the corestriction map for $K[np^{m+1}]/K[np^m]$.
It is known that the collection of Heegner cohomology classes forms an (anticyclotomic) Euler system for $T$   \cite[\S1]{kolyvagin-euler-systems}, \cite[\S3]{gross-kolyvagin}.
Then the norm compatibility of Heegner points \cite[Prop. 4.1]{castella-wan-perrin-riou-ss} with the $a_p = 0$ condition yields
$\mathrm{cores}_{m+1, m} \left( z[np^{m+1}] \right) = - z[np^m] $
for every $m \geq 1$.

\subsubsection{Construction of signed Iwasawa--Heegner cohomology classes}
We give a sketch of the construction of signed Iwasawa--Heegner  cohomology classes following \cite[\S4.1]{castella-wan-perrin-riou-ss}.

Write 
$\mathrm{Gal}(K[p^\infty]/K) \simeq \Delta \times \mathrm{Gal}(K_\infty/K)$
where $\Delta$ is the torsion subgroup of $\mathrm{Gal}(K[p^\infty]/K)$.
We write 
$L = \left( K[p^\infty] \right)^{\mathrm{Gal}(K_\infty/K)}$ and
$L_{m+1} = \left( K[p^\infty] \right)^{\mathrm{Gal}(K_\infty/K_m)}$ for each $m$.
Then $\mathrm{Gal}(L_{n+1}/K) \simeq \Delta \times \mathrm{Gal}(K_n/K)$.
Then there exists a non-negative integer $\delta$ such that $L_{m+1+\delta} = \left( K[p^m] \right)^{\Delta}$ for $m \gg 0$.
When $p$ does not divides the class number of $K$, $\delta = 0$.
We write $\widetilde{\Lambda} = \mathbb{Z}_p\llbracket \mathrm{Gal}(K[p^\infty]/K) \rrbracket$.
For any positive integer $n$ prime to $p$,
let $\mathrm{H}^1_{\mathcal{I}w}(K[n], T) = \varprojlim_{m} \mathrm{H}^1(K[np^m], T) \simeq \mathrm{H}^1(K[n], T \otimes \widetilde{\Lambda})$
where the last isomorphism follows from Shapiro's lemma. 
Then it is known that  $z[np^{m}]$ lies in the image of the corestriction map
$\mathrm{H}^1_{\mathcal{I}w}(K[n], T) \to \mathrm{H}^1(K[np^m], T)$ \cite[Lem. 4.2]{castella-wan-perrin-riou-ss} and
$\mathrm{H}^1_{\mathcal{I}w}(K[n], T)$ is free over $\Lambda$ \cite[Lem. 4.3]{castella-wan-perrin-riou-ss}.

Let $\widetilde{\omega}^{+}_m(X) = \prod_{2 \leq i \leq m, i \textrm{ even}} \Phi_i(1+X)$ and
$\widetilde{\omega}^{-}_m(X) = \prod_{2 \leq i \leq m, i \textrm{ odd}} \Phi_i(1+X)$
where $\Phi_i(X)$ is the $p^i$-th cyclotomic polynomial, and set
$\omega^{\pm}_m(X) = X \cdot \widetilde{\omega}^{\pm}_m(X)$.
We also let $\omega_m(X) = (1+X)^{p^m} - 1 = \omega^{\pm}_m(X) \cdot \widetilde{\omega}^{\mp}_m(X)$.

\begin{prop}
Let $\epsilon = \epsilon(m)$ be the sign of $(-1)^m$.
Then there exists a unique cohomology class
$z^{\epsilon}_{m}[n] = z^{\epsilon(m)}_{m}[n] \in \mathrm{H}^1(K[n], T \otimes \Lambda) / \omega^{\epsilon(m)}_m$
such that $\omega^{-\epsilon}_m(X) \cdot z^{\epsilon}_m[n]  = \mathrm{cores} (z[np^{m+1-\delta}])$
where $\mathrm{cores}$
is the corestriction map from $K[np^{m+1-\delta}]$ to $K_m[n]$ and $K_m[n]$ is the compositum of $K_m$ and $K[n]$.
Furthermore, the sequences 
$\left\lbrace (-1)^{m/2}\cdot z^{+}_m[n] \right\rbrace_{m, \textrm{ even}}$ and
$\left\lbrace (-1)^{(m+1)/2}\cdot z^{-}_m[n] \right\rbrace_{m, \textrm{ odd}}$
are compatible under the natural maps
$\mathrm{H}^1(K[n], T \otimes \Lambda)/\omega^{\epsilon}_m(X)
\to
\mathrm{H}^1(K[n], T \otimes \Lambda)/\omega^{\epsilon}_{m-2}(X) $.
\end{prop}
\begin{proof}
See \cite[Prop. 4.4]{castella-wan-perrin-riou-ss}.
\end{proof}
\begin{defn}
For each sign $\epsilon \in \lbrace \pm  \rbrace$ and a positive integer $n$ prime to $Np$,
the \textbf{$\epsilon$-Iwasawa--Heegner cohomology class
$$z^{\epsilon}_{\infty}[n] \in \mathrm{H}^1_{\mathrm{Iw}}(K[n], T) \simeq \mathrm{H}^1(K[n], T \otimes \Lambda)$$
 of conductor $n$} is defined by
$$ z^{\epsilon}_{\infty}[n] := \varprojlim_m z^{\epsilon}_m[n] \in \varprojlim_m \mathrm{H}^1(K[n], T \otimes \Lambda) / \omega^{\epsilon}_m(X)\mathrm{H}^1(K[n], T \otimes \Lambda)
 \simeq \mathrm{H}^1(K[n], T \otimes \Lambda)$$
 where $m$ runs over positive integers with $\epsilon = \mathrm{sign}((-1)^m)$.
\end{defn}
The well-definedness of $z^{\epsilon}_{\infty}[n]$ follows from that the fact $\omega^{\epsilon}_m(X)$ with $\epsilon = \mathrm{sign}((-1)^m)$ forms a basis for the topology of $\Lambda$.

\subsubsection{The signed Heegner point main conjecture}
Let $z^{\pm}_\infty \in \mathrm{H}^1(K, T \otimes \Lambda)$ be the image of $z^{\pm}_\infty[1]$ under the corestriction map
$\mathrm{H}^1(K[1], T \otimes \Lambda) \to \mathrm{H}^1(K, T \otimes \Lambda)$
and it actually lies in
$\widehat{\mathrm{Sel}}_{\pm}(K_\infty, T)$ thanks to \cite[Lem. 4.7]{castella-wan-perrin-riou-ss}.

In \cite[Conj. 4.8]{castella-wan-perrin-riou-ss}, Castella--Wan formulated the supersingular analogue of Perrin-Riou's Heegner point main conjecture \cite{perrin-riou-heegner}, in particular,  Equality (\ref{eqn:heegner-pt-main-conj-signed}), as follows:
\begin{conj}[Signed Heegner point main conjectures] \label{conj:heegner-pt-main-conj-signed}
Let $E$ be an elliptic curve of conductor $N$ over $\mathbb{Q}$ and $p \geq 5$ be a good supersingular reduction prime for $E$.
Let $K$ be an imaginary quadratic field of odd discriminant $-D_K < -4$ such that $N^-$ is square-free, $\nu(N^-)$ is even, and $p$ splits in $K/\mathbb{Q}$.
Then the following statements hold:
\begin{enumerate}
\item The signed Iwasawa--Heegner cohomology class $z^{\pm}_\infty$ is not $\Lambda$-torsion.
\item Both $\widehat{\mathrm{Sel}}_{\pm}(K_\infty, T)$ and $\mathrm{Sel}^{\pm}( K_\infty, E[p^\infty])^\vee$ 
are $\Lambda$-modules of rank one.
\item there exists a finitely generated torsion $\Lambda$-module $M^{\pm}_\infty$ such that
\begin{enumerate}
\item 
$\mathrm{Sel}^{\pm}( K_\infty, E[p^\infty])^\vee \sim \Lambda \oplus M^{\pm}_{\infty} \oplus M^{\pm}_{\infty}$, and
\item \begin{equation} \label{eqn:heegner-pt-main-conj-signed}
\mathrm{char}_{\Lambda} \left(  \dfrac{\widehat{\mathrm{Sel}}_{\pm}(K_\infty, T)}{\kappa^{\mathrm{Hg}, \pm, \infty}_1 } \right) = \mathrm{char}_{\Lambda} \left( M^{\pm}_\infty \right)
\end{equation}
 in $\Lambda$,
\end{enumerate}
where $\sim$ means a pseudo-isomorphism of  (finitely generated torsion) Iwasawa modules.
\end{enumerate}
\end{conj}
\begin{thm}[Castella--Wan] \label{thm:heegner-pt-main-conj-signed}
Let $E$ be a semi-stable elliptic curve of conductor $N$ over $\mathbb{Q}$ and $p \geq 5$ be a good supersingular reduction prime for $E$.
Under the choice of $K$ in \S\ref{subsubseec:choice-K}, the following statements hold:
\begin{enumerate}
\item The signed Iwasawa--Heegner cohomology class $z^{\pm}_\infty$ is not $\Lambda$-torsion.
\item Both $\widehat{\mathrm{Sel}}_{\pm}(K_\infty, T)$ and $\mathrm{Sel}^{\pm}( K_\infty, E[p^\infty])^\vee$ 
are $\Lambda$-modules of rank one.
\item There exists a finitely generated torsion $\Lambda$-module $M^{\pm}_\infty$ such that
\begin{enumerate}
\item 
$\mathrm{Sel}^{\pm}( K_\infty, E[p^\infty])^\vee \sim \Lambda \oplus M^{\pm}_{\infty} \oplus M^{\pm}_{\infty}$, and
\item 
$\mathrm{char}_{\Lambda} \left(  \dfrac{\widehat{\mathrm{Sel}}_{\pm}(K_\infty, T)}{\kappa^{\mathrm{Hg}, \pm, \infty}_1 } \right) = \mathrm{char}_{\Lambda} \left( M^{\pm}_\infty \right)
$
 in $\Lambda[1/p]$.
\end{enumerate}
If further $E[p]$ is ramified at every prime dividing $N^-$, then the equality in (b) holds in $\Lambda$.
\end{enumerate}
\end{thm}
\begin{proof}
It is \cite[Thm. C]{castella-wan-perrin-riou-ss}.
More precisely, (1) follows from the non-triviality of the BDP $p$-adic $L$-function $\mathscr{L}^{\mathrm{BDP}}_{\mathfrak{p}}$ \cite{hsieh-anti-rankin-selberg, burungale-non-triviality-shimura}  
and the explicit reciprocity law relating $\mathscr{L}^{\mathrm{BDP}}_{\mathfrak{p}}$ to $z^{\pm}_\infty$ \cite[Thm. 6.2]{castella-wan-perrin-riou-ss}.
We use their non-triviality results only, not the vanishing of their $\mu$-invariants.
\end{proof}
See also \cite{wan-heegner, burungale-castella-kim, castella-wan-imc-derivatives, sweeting-kolyvagin} for the recent developments on Heegner point main conjectures.

\subsection{Heegner point Kolyvagin systems}
\subsubsection{Construction of signed Iwasawa--Kolyvagin cohomology classes} \label{subsubsec:construction-kolyvagin-classes}
Let $k \geq 1$ be an integer.
Let $\mathcal{P}_k = \mathcal{P}_k(E,p,K)$ be the set of mod $p^k$ Kolyvagin primes; in other words, $\ell \in \mathcal{P}_k$ if
 $\ell$ is coprime to $2NpD_K$,
 $\ell$ is inert in $K/\mathbb{Q}$,
 $\ell +1 \equiv 0 \pmod{p^k}$,  and $a_\ell(E) \equiv 0 \pmod{p^k}$.
Let $\mathcal{N}_k$ be the set of square-free products of the primes in $\mathcal{P}_k$.
For $\ell \in \mathcal{P}_1$, let $I_\ell = (\ell +1, 1 - a_\ell(E) + \ell ) \subseteq \mathbb{Z}_p$.
For $n \in \mathcal{N}_1$, let $I_n = \sum_{\ell \vert n} I_\ell  \subseteq \mathbb{Z}_p$.

For $n \in \mathcal{N}_1$, let $G_{n} = \mathrm{Gal}( K[n]/ K[1] )$ and $\mathcal{G}_n = \mathrm{Gal}( K[n]/ K )$.
For each $\ell \in \mathcal{N}_1$, $G_{\ell}$ is a cyclic group of order $\ell+1$ and choose a generator $\sigma_\ell$ of $G_\ell$.
Then $G_{n} = \prod_{\ell \vert n}  G_{\ell}$ for $n \in \mathcal{N}_1$.
Write $D_\ell = \sum^{\ell}_{i=1} i \cdot \sigma^i_\ell \in \mathbb{Z}[G_{\ell}]$
and
$D_n = \prod_{\ell \vert n} D_\ell \in \mathbb{Z}[G_{n}]$.
Then it is known that  
$$\sum_{\varsigma \in S} \varsigma \left( D_nz^{\pm}_{\infty}[n] \right) \pmod{I_n} \in \widehat{\mathrm{Sel}}_{\pm}(K_\infty[n], T/I_nT)^{\mathcal{G}_n}$$
where $S$ is a set of generators of $\mathrm{Gal}(K[np^\infty]/K) / \mathrm{Gal}(K[np^\infty]/K[1])$.

Let 
$\kappa^{\pm, \infty}_n \in \mathrm{H}^1_{\mathrm{Iw}}(K_\infty, T/I_nT)$
be the image of $ \sum_{\varsigma \in S} \varsigma \left( D_nz^{\pm}_{\infty}[n] \right) \pmod{I_n} $
under the inverse of the restriction isomorphism  
$ \mathrm{res}^{-1} : \mathrm{H}^1_{\mathrm{Iw}}(K_\infty[n], T/I_nT)^{\mathcal{G}_n} \to \mathrm{H}^1_{\mathrm{Iw}}(K_\infty, T/I_nT)$ \cite[Lem. A.1]{castella-wan-perrin-riou-ss}.
Although $\kappa^{\pm, \infty}_n$ is ramified at primes dividing $n$, we do not know whether 
$\kappa^{\pm, \infty}_n \in \widehat{\mathrm{Sel}}_{\pm,n}(K_\infty, T/I_nT)$  in general
where $\widehat{\mathrm{Sel}}_{\pm,n}$ means the $n$-transverse signed $\Lambda$-adic Selmer structure \cite{mazur-rubin-book, howard-kolyvagin}.
However, there exists a (minimal) integer  $d$ with $0 \leq d \leq \mathrm{ord}_p(\mathrm{Tam}(E/K))$ independent of $n$ such that
$$\kappa^{\mathrm{Hg}, \pm, \infty}_n := p^d \cdot \kappa^{\pm, \infty}_n \in \widehat{\mathrm{Sel}}_{\pm,n}(K_\infty, T/I_nT),$$
which we call the \textbf{signed Iwasawa--Kolyvagin cohomology class at $n$},
and the collection of $\kappa^{\mathrm{Hg}, \pm, \infty}_n$ forms a $\Lambda$-adic Heegner point Kolyvagin system
$\ks^{\mathrm{Hg}, \pm, \infty}$
for $(T \otimes \Lambda, \mathcal{F}^{\pm}_{\Lambda}, \mathcal{N}_1)$ so that
$\ks^{\mathrm{Hg}, \pm, \infty}$
 satisfies the Kolyvagin system relation
\begin{equation} \label{eqn:kolyvagin-sysetm-relation}
 \phi^{\mathrm{fs}}_{\ell} \circ \mathrm{loc}_{\ell} \left( \kappa^{\mathrm{Hg}, \pm, \infty}_n \right)
 = \mathrm{loc}_{\ell} \left( \kappa^{\mathrm{Hg}, \pm, \infty}_{\ell n} \right)
\end{equation}
for every $\ell n \in \mathcal{N}_1$ where
$\phi^{\mathrm{fs}}_{\ell} : \mathrm{H}^1_{f}(K_\ell , T \otimes \Lambda / I_{n\ell}T \otimes \Lambda) \to \mathrm{H}^1_{/f}(K_\ell , T \otimes \Lambda / I_{n\ell}T \otimes \Lambda)$ is (the $\Lambda$-adic version of) the finite-singular isomorphism \cite[Lem. A.3]{castella-wan-perrin-riou-ss}.
Here, $\mathrm{Tam}(E/K)$ means the product of local Tamagawa factors of $E$ over $K$.
See \cite[Lem. A.2, A.3]{castella-wan-perrin-riou-ss} for details.
\begin{rem}
It is known that $d = 0$ when Heegner points come from modular curves (the $N^-=1$ case) \cite{gross-kolyvagin},  \cite[Lem. 1.7.3]{howard-kolyvagin} or there is no Tamagawa defect at primes dividing $N^+$ \cite[Assu. 1.1.1]{zanarella-howard}.
In \cite{kim-gross-zagier}, the consideration of $d$ is unfortunately missing, but there is no harm in the main result therein. It would be more complicated to formulate the refined Kolyvagin's conjecture when $d \neq 0$.
\end{rem}
For a height one prime $\mathfrak{P} \neq p\Lambda$, the specialization of $\kappa^{\mathrm{Hg}, \pm, \infty}_n$ at $\mathfrak{P}$
$$\kappa^{\mathrm{Hg}, \pm, (\mathfrak{P})}_n \in \mathrm{Sel}_{\pm, n}(K, T_{\mathfrak{P}}/I_n T_{\mathfrak{P}}) $$ 
is defined by the image of $\kappa^{\mathrm{Hg}, \pm, \infty}_n$ under the natural map
$\widehat{\mathrm{Sel}}_{\pm,n}(K_\infty, T/I_nT) \to \mathrm{Sel}_{\pm,n}(K, T_{\mathfrak{P}}/I_n T_{\mathfrak{P}}) $. 
Then the collection of $\kappa^{\mathrm{Hg}, \pm, (\mathfrak{P})}_n$ forms the \textbf{Heegner point Kolyvagin system
$\ks^{\mathrm{Hg}, \pm, (\mathfrak{P})}$}
for $(T_{\mathfrak{P}}, \mathcal{F}^{\pm}_{\mathfrak{P}}, \mathcal{N}_1)$ and satisfies the Kolyvagin system relation as the specialization of (\ref{eqn:kolyvagin-sysetm-relation}).
\subsubsection{The numerical invariants}
The numerical invariants associated to $\ks^{\mathrm{Hg}, \pm, (\mathfrak{P})}$ are denoted by
\begin{align} \label{eqn:numerical-invariants}
\begin{split}
\mathrm{ord} (\ks^{\mathrm{Hg}, \pm, (\mathfrak{P})}) & = \mathrm{min} \left\lbrace \nu(n) : n \in \mathcal{N}_1, \kappa^{\mathrm{Hg}, \pm, (\mathfrak{P})}_n \neq 0 \right\rbrace , \\
\mathrm{ord}_{\mathfrak{m}_{\mathfrak{P}}}( \kappa^{\mathrm{Hg}, \pm, (\mathfrak{P})}_n ) & = \mathrm{max} \left\lbrace j : \kappa^{\mathrm{Hg}, \pm, (\mathfrak{P})}_n \in \mathfrak{m}^j_{\mathfrak{P}} \mathrm{Sel}_{\pm, n}(K, T_{\mathfrak{P}}/I_n T_{\mathfrak{P}})  \right\rbrace , \\
\partial^{(i)}( \ks^{\mathrm{Hg}, \pm, (\mathfrak{P})} ) & = \mathrm{min} \left\lbrace \mathrm{ord}_{\mathfrak{m}_{\mathfrak{P}}}( \kappa^{\mathrm{Hg}, \pm, (\mathfrak{P})}_n ) :  n \in \mathcal{N}_1 \textrm{ with } \nu(n) = i  \right\rbrace , \\
\partial^{(\infty)}( \ks^{\mathrm{Hg}, \pm, (\mathfrak{P})}  ) & = \mathrm{min} \left\lbrace \partial^{(i)}( \ks^{\mathrm{Hg}, \pm, (\mathfrak{P})}  )  : i \geq 0 \right\rbrace .
\end{split}
\end{align}
We also define
$$d(\ks^{\mathrm{Hg}, \pm, (\mathfrak{P})}) = \mathrm{min} \left\lbrace \mathrm{ord}_{\mathfrak{m}_{\mathfrak{P}}}(\kappa^{\mathrm{Hg}, \pm, (\mathfrak{P})}_n) : n \in \mathcal{N}_k , \mathrm{dim}_{S_{\mathfrak{P}}/\mathfrak{m}_{\mathfrak{P}}} \mathrm{Sel}_{\pm, n}(K, T_{\mathfrak{P}}/\mathfrak{m}_{\mathfrak{P}} T_{\mathfrak{P}}) = 1 \right\rbrace,$$
and by definition, we have
$\partial^{(\infty)}( \ks^{\mathrm{Hg}, \pm, (\mathfrak{P})}  ) \leq d(\ks^{\mathrm{Hg}, \pm, (\mathfrak{P})})$.

\subsubsection{Derived Heegner points} \label{subsubsec:derived-heegner-points}
We continue from \S\ref{subsubsec:review-heegner-points}.
A \textbf{derived Heegner point over $K[np^m]$} is defined by 
\begin{equation} \label{eqn:P_n}
P_{np^m} = \sum_{\varsigma \in S} \varsigma( D_n y_{np^m} )
\end{equation}
where $S$ is a set of generators of $\mathrm{Gal}(K[n]/K) / \mathrm{Gal}(K[n]/K[1])$.
Denote by $[P_n]$ the image of $P_n$ in $E(K[n])/I_nE(K[n])$.
Since $D_n y_n \pmod{I_n}$ is $G_n$-invariant and $[P_n]$ is independent of $S$,
$[P_n]$ lies in $(E(K[n])/I_nE(K[n]))^{\mathcal{G}_n}$.

Let $n \in \mathcal{N}_k$ and assume $p^k\mathbb{Z}_p = I_n \mathbb{Z}_p$ for convenience.
Consider the commutative diagram
\begin{equation} \label{eqn:kummer-map-Kn}
\begin{split}
\xymatrix@=1.2em{
& &\mathrm{H}^1(K[n]/K, E)[p^k] \ar@{^{(}->}[d]^-{\mathrm{inf}} \\
E(K) /p^k  \ar@{^{(}->}[r]^-{\delta} \ar@{^{(}->}[d] & \mathrm{H}^1(K, E[p^k]) \ar@{->>}[r] \ar[d]^-{\mathrm{res}}_-{\simeq} & \mathrm{H}^1(K, E)[p^k] \ar[d]^-{\mathrm{res}} \\
\left(  E(K[n]) / p^k \right)^{\mathcal{G}_n} \ar@{^{(}->}[r]^-{\delta_n} \ar@{->>}[d]^-{\overline{\delta}_n} &  \left(\mathrm{H}^1(K[n], E[p^k]) \right)^{\mathcal{G}_n} \ar[r] & \left(\mathrm{H}^1(K[n], E)[p^k] \right)^{\mathcal{G}_n} \\
\mathrm{H}^1(K[n]/K, E)[p^k]
}
\end{split}
\end{equation}
For $n \in \mathcal{N}_k$, $\kappa^{\mathrm{Hg}}_n$ can also be explicitly defined as 
\begin{equation*}
 \kappa^{\mathrm{Hg}}_n := \mathrm{res}^{-1} \left( \delta_n([P_n]) \right) \in \mathrm{H}^1(K, E[p^k]) 
\end{equation*}
and it is represented by
1-cocycle
$\sigma \mapsto  \left( P_n / p^k \right)^\sigma - P_n / p^k - \dfrac{1}{p^k} \left( P^\sigma_n - P_n \right)$
where $\sigma \in \mathrm{Gal}(\overline{K}/K)$.
We also observe that
$\overline{\delta}_n ([P_n])  \in \mathrm{H}^1(K[n]/K, E(K[n]))[p^k]$
is represented by
1-cocycle
$\overline{\sigma} \mapsto - \dfrac{1}{p^k} \left( P^{\overline{\sigma}}_n - P_n \right)$
where $\overline{\sigma} \in \mathrm{Gal}(K[n]/K)$.
See \cite[Lem. 4.1 and Cor. 4.2]{mccallum-kolyvagin} and \cite[pg. 400]{ciperiani-wiles} for example.

The following corollary is a simple observation but is essential for our purpose.
\begin{cor} \label{cor:unramified-everywhere}
Let $n \in \mathcal{N}_k$.
If $ \kappa^{\mathrm{Hg}}_{n}$ is non-trivial and $\nu(n) = \mathrm{ord} (\ks^{\mathrm{Hg}})$,
then $\kappa^{\mathrm{Hg}}_{n}$ is unramified everywhere, i.e. $\kappa^{\mathrm{Hg}}_{n} \in \mathrm{Sel}(K, E[I_n])$.
\end{cor}
\begin{proof}
It immediately follows from the definition of $\mathrm{ord} (\ks^{\mathrm{Hg}})$ and the specialization of (\ref{eqn:kolyvagin-sysetm-relation}) to the augmentation ideal.
\end{proof}

\section{A proof of Kolyvagin's conjecture} \label{sec:kolyvagin-conjecture}
We prove Theorem \ref{thm:main}.
\subsection{The rigidity of Heegner point Kolyvagin systems}
For $n \in \mathcal{N}_k$,
let 
$$\lambda^{\pm, (k)}(n) =  \mathrm{length}(M^{\pm, (k)}(n))$$
where
$\mathrm{Sel}_{\pm,n}(K, T_{\mathfrak{P}} / \mathfrak{m}^{k}_{\mathfrak{P}} T_{\mathfrak{P}}) \simeq 
\left( S_{\mathfrak{P}} / \mathfrak{m}^{k}_{\mathfrak{P}} \right)^{\epsilon} \oplus M^{\pm, (k)}(n)^{\oplus 2} $
with $\epsilon \in \lbrace 0, 1 \rbrace$
and the last isomorphism follows from  \cite[Lem. 1.5.1]{howard-kolyvagin}. 
It should be noted that the argument for the ordinary setting  applies equally to the supersingular one since the signed Selmer structure is also self-dual.
The following rigidity result is essential for our purpose.
\begin{lem}[Zanarella] \label{lem:rigidity}
Let $\ks^{\pm} \in \KS(T_{\mathfrak{P}}, \mathcal{F}^{\pm}_{\mathfrak{P}})$ be a non-trivial Kolyvagin system. There exist constants 
$$d(\ks^{\pm}) \in \mathbb{Z}_{\geq 0} \cup \lbrace \infty \rbrace$$ 
such that
$\kappa^{(k)}_n  = \mathfrak{m}^{d(\ks^{\pm})+ \lambda^{\pm, (k)}(n)}_{\mathfrak{P}} \mathrm{Sel}_{\pm, n}(K, T_{\mathfrak{P}} / \mathfrak{m}^{k}_{\mathfrak{P}} T_{\mathfrak{P}}) $ for $n \in \mathcal{N}_{2k-1}$ and
 $d(\ks^{\pm})$ is independent of $k > 0$.
\end{lem}
\begin{proof}
This is \cite[Lem. 2.3.1]{zanarella-howard}, which strengthens \cite[Lem. 1.6.4]{howard-kolyvagin}.
In its proof, it is also shown that
$d(\ks^{\pm}) = \mathrm{min} \left\lbrace \mathrm{ord}_{\mathfrak{m}_{\mathfrak{P}}}(\kappa^{\pm}_n) : n \in \mathcal{N}_k , \mathrm{dim}_{S_{\mathfrak{P}}/\mathfrak{m}_{\mathfrak{P}}} \mathrm{Sel}_{\pm, n}(K, T_{\mathfrak{P}}/\mathfrak{m}_{\mathfrak{P}} T_{\mathfrak{P}}) = 1 \right\rbrace $.
\end{proof}
\begin{thm}[Howard, Zanarella] \label{thm:exact-bound}
Let $\ks^{\pm} \in \KS(T_{\mathfrak{P}}, \mathcal{F}^{\pm}_{\mathfrak{P}})$ be a Kolyvagin system with
$\kappa^{\pm}_1 \neq 0 $.
Then $\mathrm{Sel}_{\pm}(K, T_{\mathfrak{P}})$ is a free $S_{\mathfrak{P}}$-module of rank one, and there exists a finite $S_{\mathfrak{P}}$-module $M^{\pm}_{\mathfrak{P}}$
such that
\[
\xymatrix{
\mathrm{Sel}^{\pm}(K, (T_{\mathfrak{P}})^* )  \simeq 
\Psi_{\mathfrak{P}} / S_{\mathfrak{P}} \oplus M^{\pm, \oplus 2}_{\mathfrak{P}}  , &
\mathrm{length}_{ S_{\mathfrak{P}} } (M^{\pm}_{\mathfrak{P}})  = \mathrm{length}_{ S_{\mathfrak{P}} } \left( \frac{ \mathrm{Sel}_{\pm}(K, T_{\mathfrak{P}} ) }{ \kappa^{\pm}_1 } \right) - d(\ks^{\pm}) .
}
\]
\end{thm}
\begin{proof}
This follows from \cite[Thm. 1.6.1]{howard-kolyvagin} and \cite[Thm. 2.3.6]{zanarella-howard}.
\end{proof}

\subsection{The proof of Kolyvagin's conjecture}
We are now ready to describe a precise connection between the signed Heegner point main conjecture and Kolyvagin's conjecture. This is the supersingular variant of \cite[Thm. 6.5]{kim-gross-zagier}.
\begin{thm} \label{thm:non-triviality-heegner}
We keep all the assumptions in $\S$\ref{subsec:working-hypotheses}.
Let $\mathfrak{P}$ be a height one prime of $\Lambda$.
Then the following statements are equivalent.
\begin{enumerate}
\item  $ \ks^{\mathrm{Hg}, \pm, \infty}$ is non-trivial modulo $\mathfrak{P}$.
\item $\mathrm{ord}_{\mathfrak{P}} \left( \mathrm{char}_{\Lambda} \left(  \dfrac{\widehat{\mathrm{Sel}}_{\pm}(K_\infty, T)}{ \kappa^{\mathrm{Hg}, \pm, \infty}_1 } \right) \right) = \mathrm{ord}_{\mathfrak{P}} \left( \mathrm{char}_{\Lambda} \left( M^{\pm}_{\infty} \right) \right)$.
\end{enumerate}
\end{thm}
\begin{proof}
Let $f_{\Lambda}$ be a generator of $\mathrm{char}_{\Lambda} \left(  \dfrac{\widehat{\mathrm{Sel}}_{\pm}(K_\infty, T)}{ \kappa^{\mathrm{Hg}, \pm, \infty}_1 }\right) $. 
We first suppose that $\mathfrak{P} \neq p\Lambda$ and write $\mathfrak{P} = (g)$ where $g \in \Lambda$.
Define $\mathfrak{P}_M  = (g+p^M)\Lambda$.
By taking $M$ sufficiently large, we may assume the following statements:
\begin{itemize}
\item $\mathfrak{P}_M$ is a height one prime ideal of  $\Lambda$.
\item $\Lambda / \mathfrak{P}_M \simeq \Lambda / \mathfrak{P}$ as rings.
\item $\mathfrak{P}_M$ and $f_{\Lambda}$ are relatively prime, so we have $\kappa^{\mathrm{Hg}, \pm, (\mathfrak{P}_M)}_1 \neq 0$.
\item Both natural maps
\[
\xymatrix{
\widehat{\mathrm{Sel}}_{\pm}(K_\infty, T)/  \mathfrak{P}_M \to \mathrm{Sel}_{\pm}(K, T_{\mathfrak{P}_M}),
& \mathrm{Sel}^{\pm}(K, ( T_{\mathfrak{P}_M} )^* )  \to \mathrm{Sel}^{\pm}(K_\infty, E[p^\infty])[\mathfrak{P}_M]
}
\]
have finite kernel and cokernel which are bounded by a constant depending only on $[S_{\mathfrak{P}_M} : \Lambda / \mathfrak{P}_M]$
but not on $\mathfrak{P}_M$ itself
where $S_{\mathfrak{P}_M}$ is the integral closure of $\Lambda / \mathfrak{P}_M$ \cite[Prop. 2.2.8]{howard-kolyvagin}.
\end{itemize}
Following the argument of \cite[Thm. 5.3.10]{mazur-rubin-book},
the equality 
$(\mathfrak{P}_M, \mathfrak{P}^n ) = 
(\mathfrak{P}_M, p^{Mn} )$
implies that
\begin{align*}
\mathrm{length}_{\mathbb{Z}_p} \left(  \mathrm{Sel}_{\pm}(K, T_{\mathfrak{P}_M} ) / S_{\mathfrak{P}_M} \kappa^{\mathrm{Hg}, \pm, (\mathfrak{P}_M)}_1 \right) & = 
\mathrm{length}_{\mathbb{Z}_p} \left(  \Lambda / (f_{\Lambda} , \mathfrak{P}_M \right) \\
& = \mathrm{length}_{\mathbb{Z}_p} \left(  \Lambda / (  \mathfrak{P}^{ \mathrm{ord}_{\mathfrak{P}} (  f_{\Lambda})}, \mathfrak{P}_M \right) \\
& = M \cdot \mathrm{rk}_{\mathbb{Z}_p}  ( \Lambda / \mathfrak{P}_M )  \cdot \mathrm{ord}_{\mathfrak{P}} (  f_{\Lambda} )
\end{align*}
up to $O(1)$ as $M$ varies. Here, $O(1)$ means that the differences are constant independent of $M$.
Similarly, we have
\begin{align*}
2 \cdot \mathrm{length}_{\mathbb{Z}_p} (M^{\pm}_{\mathfrak{P}_M}) & = \mathrm{length}_{\mathbb{Z}_p} \left( \mathrm{Sel}^{\pm}(K,  ( T_{\mathfrak{P}_M} )^* )_{/\mathrm{div}} \right) \\
& = \mathrm{length}_{\mathbb{Z}_p} \left(  ( \mathrm{Sel}^{\pm}(K_\infty,  E[p^\infty])^\vee / \mathfrak{P}_M )_{\mathbb{Z}_p\textrm{-}\mathrm{tors}} \right) \\
& = M \cdot \mathrm{rk}_{\mathbb{Z}_p}  ( \Lambda / \mathfrak{P}_M )  \cdot \mathrm{ord}_{\mathfrak{P}} \left( \mathrm{char}_{\Lambda} \left( ( \mathrm{Sel}^{\pm}( K_\infty, E[p^\infty])^\vee)_{\Lambda\textrm{-}\mathrm{tors}}  \right) \right)
\end{align*}
up to $O(1)$ as $M$ varies where $M^{\pm}_{\mathfrak{P}_M} = M^{\pm}_\infty \otimes_{\Lambda} S_{\mathfrak{P}_M}$.
We also have equality
\begin{align*}
\mathrm{length}_{S_{\mathfrak{P}_M}} \left(  \mathrm{Sel}_{\pm}(K, T_{\mathfrak{P}_M} ) / S_{\mathfrak{P}_M} \kappa^{\mathrm{Hg}, \pm, (\mathfrak{P}_M)}_1 \right) = \partial^{(0)} ( \ks^{\mathrm{Hg}, \pm , (\mathfrak{P}_M)} ) .
\end{align*}
Although we consider Kolyvagin systems at height one prime $\mathfrak{P}_M$, the corresponding Selmer structure is self-dual thanks to \cite[Lem. 2.1.1]{howard-kolyvagin}. 
Since $\kappa^{\mathrm{Hg}, \pm, (\mathfrak{P}_M)}_1 \neq 0$, we have
\begin{align*}
\mathrm{cork}_{ S_{\mathfrak{P}_M} }\mathrm{Sel}^{\pm}(K, (T_{\mathfrak{P}_M})^*) & = 1, \\
\mathrm{Sel}^{\pm}(K, (T_{\mathfrak{P}_M})^*)_{/\mathrm{div}} & \simeq \bigoplus_{i \geq 1} \left( S_{\mathfrak{P}_M} /  \mathfrak{m}^{a_i}_{ \mathfrak{P}_M }  S_{\mathfrak{P}_M} \right)^{\oplus 2}
\end{align*}
where $\mathfrak{m}_{ \mathfrak{P}_M }$ is the maximal ideal of $S_{\mathfrak{P}_M}$ and $a_1 \geq  a_2 \geq  \cdots$ thanks to \cite[Thm. 1.6.1]{howard-kolyvagin}.
In particular, for any $k \geq 1$ such that $\kappa^{\mathrm{Hg}, \pm, (\mathfrak{P}_M)}_1 \neq 0$ in $S_{\mathfrak{P}_M} /\mathfrak{m}^{k}_{ \mathfrak{P}_M} S_{\mathfrak{P}_M}$, we have
\begin{align*}
\mathrm{Sel}^{\pm}(K, (T_{\mathfrak{P}_M})^*[\mathfrak{m}^{k}_{ \mathfrak{P}_M}])  & \simeq  S_{\mathfrak{P}_M} /\mathfrak{m}^{k}_{ \mathfrak{P}_M} S_{\mathfrak{P}_M} \oplus M^{\pm, (k)}_{\mathfrak{P}_M} \oplus M^{\pm, (k)}_{\mathfrak{P}_M}  \\
& \simeq S_{\mathfrak{P}_M} /\mathfrak{m}^{k}_{ \mathfrak{P}_M} S_{\mathfrak{P}_M} \oplus \bigoplus_{i \geq 1} \left( S_{\mathfrak{P}_M} /  \mathfrak{m}^{a_i}_{ \mathfrak{P}_M }  S_{\mathfrak{P}_M} \right)^{\oplus 2}
\end{align*}
with $\mathrm{length}_{S_{\mathfrak{P}_M}} M^{\pm, (k)}_{\mathfrak{P}_M} < k$ following the proof of \cite[Thm. 1.6.1]{howard-kolyvagin}.
Denote by $$\ks^{\mathrm{Hg}, \pm, (\mathfrak{P}_M), 2k-1} = \left\lbrace \kappa^{\mathrm{Hg}, \pm, (\mathfrak{P}_M)}_n : n \in \mathcal{N}_{2k-1} \right\rbrace$$
 the Heegner point Kolyvagin system for $(T_{\mathfrak{P}_M}/ \mathfrak{m}^{k}_{ \mathfrak{P}_M} T_{\mathfrak{P}_M}, \mathcal{F}^{\pm}_{\mathfrak{P}_M}, \mathcal{N}_{2k-1})$.
By the rigidity of Heegner point Kolyvagin systems (Lemma \ref{lem:rigidity}), for $n \in \mathcal{N}_{2k-1}$, there exist unique integers $\delta^{\mathrm{Hg}, \pm}_{\mathfrak{P}_M}(k)$, independent of $n$, such that
\begin{equation} \label{eqn:rigidity-kolyvagin}
\left\langle \kappa^{\mathrm{Hg}, \pm, (\mathfrak{P}_M)}_n \right\rangle  = \mathfrak{m}^{\lambda^{(k)}(n) + \delta^{\mathrm{Hg}, \pm}_{\mathfrak{P}_M}(k) }_{ \mathfrak{P}_M}\mathrm{Sel}_{\mathcal{F}^{\pm}_{\mathfrak{P}_M}(n)}(K, T_{\mathfrak{P}_M}/ \mathfrak{m}^{k}_{ \mathfrak{P}_M} T_{\mathfrak{P}_M})
\end{equation}
where $\lambda^{(k)}(n) = \mathrm{length}_{ S_{\mathfrak{P}_M} }  M^{\pm,(k)}_{\mathfrak{P}_M}(n)$ and
$$\mathrm{Sel}_{\mathcal{F}^{\pm, *}_{\mathfrak{P}_M}(n)}(K, (T_{\mathfrak{P}_M})^*[\mathfrak{m}^{k}_{ S_{\mathfrak{P}_M}}])   \simeq  S_{\mathfrak{P}_M} /\mathfrak{m}^{k}_{ \mathfrak{P}_M} S_{\mathfrak{P}_M} \oplus M^{\pm, (k)}_{\mathfrak{P}_M}(n) \oplus M^{\pm, (k)}_{\mathfrak{P}_M}(n).$$
The restriction of Kolyvagin primes to $\mathcal{N}_{2k-1}$ is essential here and 
$$\delta^{\mathrm{Hg}, \pm}_{\mathfrak{P}_M}(k) = d(\ks^{\mathrm{Hg}, \pm, (\mathfrak{P}_M)})$$
 is also independent of $k$ (Theorem \ref{thm:exact-bound}). 
By taking $k$ sufficiently large, we have
\begin{align*}
\mathrm{length}_{S_{\mathfrak{P}_M}} M_{\mathfrak{P}_M} & = \sum_{i \geq 1} a^{\pm}_i \\
& = \partial^{(0)} ( \ks^{\mathrm{Hg}, \pm, (\mathfrak{P}_M)} ) - \delta^{\mathrm{Hg}, \pm}_{\mathfrak{P}_M}(k) .
\end{align*}
Combining all the above computations, we obtain equalities
\begin{align} \label{eqn:heegner-non-triviality-1}
\begin{split}
& 2 \cdot M \cdot \mathrm{rk}_{\mathbb{Z}_p}  ( S_{\mathfrak{P}_M} )  \cdot \mathrm{ord}_{\mathfrak{P}} (  f_{\Lambda} ) \\
& = 2 \cdot \mathrm{length}_{\mathbb{Z}_p} \left(  \mathrm{Sel}_{\pm}(K, T_{\mathfrak{P}_M} ) / S_{\mathfrak{P}_M} \kappa^{\mathrm{Hg}, \pm, (\mathfrak{P}_M)}_1 \right) \\ 
& =  2 \cdot\dfrac{\mathrm{rk}_{\mathbb{Z}_p}  ( S_{\mathfrak{P}_M} ) }{e(S_{\mathfrak{P}_M} / \mathbb{Z}_p)} \cdot  \mathrm{length}_{S_{\mathfrak{P}_M}} \left(  \mathrm{Sel}_{\pm}(K, T_{\mathfrak{P}_M} ) / S_{\mathfrak{P}_M} \kappa^{\mathrm{Hg}, \pm, (\mathfrak{P}_M)}_1 \right) \\
& = 2 \cdot\dfrac{\mathrm{rk}_{\mathbb{Z}_p}  ( S_{\mathfrak{P}_M} ) }{e(S_{\mathfrak{P}_M} / \mathbb{Z}_p)} \cdot \partial^{(0)} ( \ks^{\mathrm{Hg}, \pm, (\mathfrak{P}_M)} ) \\
& \geq  2 \cdot\dfrac{\mathrm{rk}_{\mathbb{Z}_p}  ( S_{\mathfrak{P}_M} ) }{e(S_{\mathfrak{P}_M} / \mathbb{Z}_p)} \cdot \left( \partial^{(0)} ( \ks^{\mathrm{Hg}, \pm, (\mathfrak{P}_M)} ) -  \delta^{\mathrm{Hg}, \pm}_{\mathfrak{P}_M}(k) \right) \\
& = \dfrac{\mathrm{rk}_{\mathbb{Z}_p}  ( S_{\mathfrak{P}_M} ) }{e(S_{\mathfrak{P}_M} / \mathbb{Z}_p)} \cdot  \mathrm{length}_{S_{\mathfrak{P}_M}} \left( \mathrm{Sel}^{\pm}(K, (T_{\mathfrak{P}_M})^*)_{/\mathrm{div}} \right)  \\
& =  \mathrm{length}_{\mathbb{Z}_p} \left( \mathrm{Sel}^{\pm}(K, (T_{\mathfrak{P}_M})^*)_{/\mathrm{div}} \right)
\end{split}
\end{align}
up to $O(1)$ as $M$ varies, and
\begin{align} \label{eqn:heegner-non-triviality-2}
\begin{split}
& \mathrm{length}_{\mathbb{Z}_p} \left( \mathrm{Sel}^{\pm}(K,  ( T_{\mathfrak{P}_M} )^* )_{/\mathrm{div}} \right) \\
& = M \cdot \mathrm{rk}_{\mathbb{Z}_p}  ( S_{\mathfrak{P}_M} )   \cdot \mathrm{ord}_{\mathfrak{P}} \left( \mathrm{char}_{\Lambda} \left( ( \mathrm{Sel}^{\pm}( K_\infty, E[p^\infty])^\vee)_{\Lambda\textrm{-tors}}  \right) \right)
\end{split}
\end{align}
up to $O(1)$ as $M$ varies again.

We prove (2) $\Rightarrow$ (1) first.
Suppose that
$$\mathrm{ord}_{\mathfrak{P}} \left( \mathrm{char}_{\Lambda} \left( ( \mathrm{Sel}^{\pm}( K_\infty, E[p^\infty])^\vee)_{\Lambda\textrm{-tors}}  \right) \right) 
=2  \cdot \mathrm{ord}_{\mathfrak{P}} (  f_{\Lambda} ) .$$
Combining (\ref{eqn:heegner-non-triviality-1}) with  (\ref{eqn:heegner-non-triviality-2}), 
the inequality above becomes an equality, so $\delta^{\mathrm{Hg}, \pm}_{\mathfrak{P}_M}(k)$ is also a constant  as $M$ varies.
Since 
\begin{equation} \label{eqn:congruences-p-M}
\ks^{\mathrm{Hg}, \pm, (\mathfrak{P})} \equiv \ks^{\mathrm{Hg}, \pm, (\mathfrak{P}_M)} \pmod{p^M}
\end{equation}
for every $M \geq 1$,
we obtain $\delta^{\mathrm{Hg}, \pm}_{\mathfrak{P}}(k) = \delta^{\mathrm{Hg}, \pm}_{\mathfrak{P}_M}(k) <\infty$  by taking $M > k$, so $ \ks^{\mathrm{Hg}, \pm, (\mathfrak{P})}$ is also non-trivial.

Now we prove (1) $\Rightarrow$ (2).
Suppose that $ \ks^{\mathrm{Hg}, \pm, (\mathfrak{P})}$ is non-trivial, so $ \delta^{\mathrm{Hg}, \pm}_{\mathfrak{P}}(k) < \infty$. 
By using the congruence (\ref{eqn:congruences-p-M}) again, $\delta^{\mathrm{Hg}, \pm}_{\mathfrak{P}_M}(k)$ is bounded as $M$ varies, so the inequality above becomes an equality.
Combining (\ref{eqn:heegner-non-triviality-1}) with  (\ref{eqn:heegner-non-triviality-2}), we obtain
$$\mathrm{ord}_{\mathfrak{P}} \left( \mathrm{char}_{\Lambda} \left( ( \mathrm{Sel}^{\pm}( K_\infty, E[p^\infty])^\vee)_{\Lambda\textrm{-tors}}  \right) \right) 
=2  \cdot \mathrm{ord}_{\mathfrak{P}} (  f_{\Lambda} ) .$$

When $\mathfrak{P} = p\Lambda$, the same argument works by taking $\mathfrak{P}_M = X^M + p$.
\end{proof}
We are now ready to obtain Theorem \ref{thm:main}.
For notational convenience, we focus only on the plus Selmer groups here.
Since we have $\ks^{\mathrm{Hg}, +, (X\Lambda)} = \ks^{\mathrm{Hg}}$ (Proposition \ref{prop:specializations-Selmer}), we get the conclusion.
\begin{cor} \label{cor:non-triviality-heegner}
Kolyvagin's conjecture holds under our working hypotheses in \S\ref{subsec:working-hypotheses}.
\end{cor}
\begin{proof}
It follows from Theorems \ref{thm:heegner-pt-main-conj-signed} and \ref{thm:non-triviality-heegner}.
\end{proof}

\section{Annihilating Tate--Shafarevich groups via Heegner points and another local-global principle} \label{sec:applications}
We study the arithmetic applications of Kolyvagin's conjecture to elliptic curves of arbitrary rank
\cite{kolyvagin-selmer, jetchev-lauter-stein,  stein-generalizing-gross-zagier, wei-zhang-mazur-tate} and give the proof of Theorem \ref{thm:main-finiteness} by applying these applications.
In the end, we discuss the local analogue of Theorem \ref{thm:main-finiteness}.
\subsection{Kolyvagin's Selmer structure and construction theorems} \label{subsec:kolyvagin-structure-construction-theorems}
Let $\mathrm{Sel}(K, E[p^\infty])^{\pm}$ be the eigenspace with eigenvalue $\pm 1$ of $\mathrm{Sel}(K, E[p^\infty])$ with respect to the complex conjugation, respectively.
Since $p >2$, we have decomposition
$\mathrm{Sel}(K, E[p^\infty]) = \mathrm{Sel}(K, E[p^\infty])^{+} \oplus \mathrm{Sel}(K, E[p^\infty])^{-} $.
Write $r_p = \mathrm{cork}_{\mathbb{Z}_p} \mathrm{Sel}(K, E[p^\infty])$ and $r^{\pm}_p = \mathrm{cork}_{\mathbb{Z}_p} \mathrm{Sel}(K, E[p^\infty])^{\pm}$.
Recall that we denote by  $w(E/\mathbb{Q})$ the root number of $E$ over $\mathbb{Q}$.
\begin{thm}[Kolyvagin] \label{thm:kolyvagin-vanishing-order}
Under our working hypotheses in \S\ref{subsec:working-hypotheses}, we have:
\begin{enumerate}
\item $r^+_p = \mathrm{ord}(\ks^{\mathrm{Hg}}) +1$.
\item $\mathrm{ord}(\ks^{\mathrm{Hg}}) - r^-_p \geq 0$ and is even.
\end{enumerate}
\end{thm}
\begin{proof}
See \cite[Thm. 4]{kolyvagin-selmer} and also \cite[Thms. 1.2 and 11.2]{wei-zhang-mazur-tate}.
Note that $r^+_p > r^-_p$ by our working hypotheses.
\end{proof}
By Theorem \ref{thm:kolyvagin-vanishing-order}, we have  
$\mathrm{ord}(\ks^{\mathrm{Hg}}) - r^{-}_p  = \vert r^{+}_p - r^{-}_p \vert - 1 $ since  $w(E/\mathbb{Q}) \cdot (-1)^{\mathrm{ord}(\ks^{\mathrm{Hg}}) } = -1$.
In particular, 
 $\mathrm{ord}(\ks^{\mathrm{Hg}}) - r^{-}_p = 0$ if and only if
$\vert r^{+}_p - r^{-}_p \vert =  1$.
Write
\begin{equation} \label{eqn:structure-selmer-before}
\mathrm{Sel}(K, E[p^\infty])^{\pm}_{/\mathrm{div}} \simeq \bigoplus_{i \geq 1} \left( \mathbb{Z} / p^{a^{\pm}_i} \mathbb{Z} \right)^{\oplus 2}
\end{equation}
where $a^{\pm}_1 \geq  a^{\pm}_2 \geq  \cdots$.
\begin{thm}[Kolyvagin] \label{thm:structure-kolyvagin}
Under our working hypotheses in \S\ref{subsec:working-hypotheses}, $a^{\pm}_i$'s in (\ref{eqn:structure-selmer-before}) are determined by $\ks^{\mathrm{Hg}}$ as follows:
\begin{align*}
a^{+}_i & = \partial^{(\mathrm{ord}(\ks^{\mathrm{Hg}}) + 2i -1)} (\ks^{\mathrm{Hg}}) - \partial^{(\mathrm{ord}(\ks^{\mathrm{Hg}}) + 2i)} (\ks^{\mathrm{Hg}} ) , \\
a^{-}_{i + \vert r^{+}_p - r^{-}_p \vert - 1 } & = \partial^{(\mathrm{ord}(\ks^{\mathrm{Heeg}}) + 2i -2)} (\ks^{\mathrm{Hg}}) - \partial^{(\mathrm{ord}(\ks^{\mathrm{Hg}}) + 2i-1)} (\ks^{\mathrm{Hg}} )
\end{align*}
for $i \geq 1$.
\end{thm}
\begin{proof}
See \cite[Thm. 1]{kolyvagin-selmer} and also \cite[Rem. 18]{wei-zhang-mazur-tate}.
\end{proof}
In order to obtain the above structure theorems, Kolyvagin first proved the following \emph{construction} theorem as an ingredient.
\begin{thm}[Kolyvagin] \label{thm:construction-rough}
Under our working hypotheses in \S\ref{subsec:working-hypotheses}, $\mathrm{Sel}(\mathbb{Q}, E[p^\infty])$ is generated by Kolyvagin cohomology classes. More precisely, we have inclusion
$$\mathrm{Sel}(\mathbb{Q}, E[p^\infty]) \subseteq \left\langle \kappa^{\mathrm{Hg}}_n : n \in \mathcal{N}_1 \right\rangle \subseteq \mathrm{H}^1(K, E[p^\infty])$$
where $\kappa^{\mathrm{Hg}}_n \in \mathrm{H}^1(K, E[I_n]) \subseteq \mathrm{H}^1(K, E[p^\infty])$.
\end{thm}
\begin{proof}
See \cite[Thms. 2 and 3]{kolyvagin-selmer} and \cite[Thm. 11.2 and Cor. 11.3]{wei-zhang-mazur-tate}.
\end{proof}
For the construction of the maximal divisible subgroup of the Selmer group, Kolyvagin gave a precise description, which was overlooked for a long time.
\begin{thm}[Kolyvagin] \label{thm:kolyvagin-construction-theorem-original}
We keep our working hypotheses in \S\ref{subsec:working-hypotheses}.
Let $k \geq 1$ be an integer.
Then there exist $2\cdot \mathrm{ord}(\ks^{\mathrm{Hg}}) +1$ Kolyvagin primes 
$$\ell_1, \ell_2,  \cdots , \ell_{2\cdot \mathrm{ord}(\ks^{\mathrm{Hg}})},  \ell_{2\cdot \mathrm{ord}(\ks^{\mathrm{Hg}}) +1} \in \mathcal{P}_{k+\partial^{(\mathrm{ord}(\ks^{\mathrm{Hg}}))} (\ks^{\mathrm{Hg}})}$$
such that
$$\mathrm{ord}_p \left(\mathrm{loc}_{\ell_{\mathrm{ord}(\ks^{\mathrm{Hg}})+i}} ( \kappa^{\mathrm{Hg}}_{n'_i} ) \right) = \partial^{(\mathrm{ord}(\ks^{\mathrm{Hg}}))} (\ks^{\mathrm{Hg}})$$
in $\mathrm{H}^1_f(K_{\ell_{\mathrm{ord}(\ks^{\mathrm{Hg}})+i}}, E[p^{k+\partial^{(\mathrm{ord}(\ks^{\mathrm{Hg}}))} (\ks^{\mathrm{Hg}})}])$
where 
$$n'_i =  \prod^{\mathrm{ord}(\ks^{\mathrm{Hg}})+i-1}_{j=i} \ell_j $$  
for $1 \leq i \leq \mathrm{ord}(\ks^{\mathrm{Hg}})+1$ so that $\nu(n'_i) = \mathrm{ord}(\ks^{\mathrm{Hg}})$.
Furthermore, 
the submodule 
\begin{align*}
\mathrm{Hg}_{k} & = \left\langle \kappa^{\mathrm{Hg}}_{n'_i} : i = 1, \cdots , \mathrm{ord}(\ks^{\mathrm{Hg}})+1 \right\rangle \\
& \subseteq p^{\partial^{(\mathrm{ord}(\ks^{\mathrm{Hg}}))} (\ks^{\mathrm{Hg}}) }\mathrm{Sel}(\mathbb{Q}, E[p^{k+\partial^{(\mathrm{ord}(\ks^{\mathrm{Hg}}))} (\ks^{\mathrm{Hg}}) }]) \subseteq \mathrm{H}^1(\mathbb{Q}, E[p^{k}])
\end{align*}
is isomorphic to $(\mathbb{Z}/p^k\mathbb{Z})^{\oplus r^+_p} = (\mathbb{Z}/p^k\mathbb{Z})^{\oplus (\mathrm{ord}(\ks^{\mathrm{Hg}})+1)}$.
\end{thm}
\begin{proof}
This is \cite[Thm. 3]{kolyvagin-selmer} with $p_0=1$.
When $k =1$ (and $\partial^{(\mathrm{ord}(\ks^{\mathrm{Hg}}))} (\ks^{\mathrm{Hg}})=0$), it is exactly the triangulation of $\mathrm{Sel}(\mathbb{Q}, E[p])$ in \cite[Lem. 8.4]{wei-zhang-mazur-tate}.
Indeed, the straightforward generalization of the proof of \cite[Lem. 8.4]{wei-zhang-mazur-tate} to the mod $p^k$ setting yields the conclusion.
When $k \gg 0$, it gives the triangulation of $\mathrm{Sel}(\mathbb{Q}, E[p^\infty])_{\mathrm{div}}[p^k]$
where $\mathrm{Sel}(\mathbb{Q}, E[p^\infty])_{\mathrm{div}}$ is the maximal divisible subgroup of $\mathrm{Sel}(\mathbb{Q}, E[p^\infty])$.
\end{proof}
\begin{rem} \label{rem:tate-shafarevich-L}
Since $\nu(n'_i) = \mathrm{ord}(\ks^{\mathrm{Hg}})$, each $\kappa^{\mathrm{Hg}}_{n'_i}$ in Theorem \ref{thm:kolyvagin-construction-theorem-original} is unramified everywhere (Corollary \ref{cor:unramified-everywhere}). 
Varying $k$ in Theorem \ref{thm:kolyvagin-construction-theorem-original}, we are able to construct
$\mathrm{Sel}(\mathbb{Q}, E[p^\infty])_{\mathrm{div}}$  in terms of $\ks^{\mathrm{Hg}}$; more precisely, we have
$$\mathrm{Sel}(\mathbb{Q}, E[p^\infty])_{\mathrm{div}} = \bigcup_{k\geq 1} \mathrm{Hg}_k $$
where the union is taken under the embedding $\mathrm{H}^1(\mathbb{Q}, E[p^{k+ \partial^{(\mathrm{ord}(\ks^{\mathrm{Hg}}))} (\ks^{\mathrm{Hg}}) }]) \hookrightarrow \mathrm{H}^1(\mathbb{Q}, E[p^{\infty}])$ under the large image assumption.
 The variation of $k$ yields (infinitely many) relations among the Kolyvagin cohomology classes in $\bigcup_{k\geq 1} \mathrm{Hg}_k $.
Denote by $L$ the compositum of $K[n'_i]$ for every $i =1 , \cdots , \mathrm{ord}(\ks^{\mathrm{Hg}})+1$ and every $k \geq 1$.
The arithmetic of $E(L)$ encodes the finiteness of $\sha(E/\mathbb{Q})[p^\infty]$ since every element of $\sha(E/\mathbb{Q})[p^\infty]_{\mathrm{div}}$ splits over $L$.
In other words, we have
$$\sha(E/\mathbb{Q})[p^\infty]_{\mathrm{div}} \subseteq \mathrm{H}^1(L/K, E(L))[p^\infty].$$
This inclusion can be viewed as a new type of ``upper bound" of $\sha(E/\mathbb{Q})[p^\infty]_{\mathrm{div}}$ although it comes from the construction of elements of Selmer groups.
By using this upper bound, we would like to study the $p$-divisibility question for $\sha(E/\mathbb{Q})[p^\infty]$ in $\mathrm{H}^1(L/K, E(L))[p^\infty]$ in a near future; the same question for $\sha(E/\mathbb{Q})[p^\infty]$ in $\mathrm{H}^1(K, E)[p^\infty]$ is treated extensively in  \cite{ciperiani-stix-WC1,ciperiani-stix-WC2}.
\end{rem}
See also \cite{wuthrich-self-points, hatton-kolyvagin-derivatives, radicevic-explict-realization} for the constructions of elements of Tate--Shafarevich groups from \emph{different} algebraic points.

\subsection{Proof of Theorem \ref{thm:main-finiteness}} \label{subsec:proof-theorem-main-finiteness}

\subsubsection{$(1) \Rightarrow (2)$}
Suppose that $\sha(E/\mathbb{Q})[p^\infty]$ is finite.
Since we assume the Galois image is large, we have
that
$\mathrm{Sel}(\mathbb{Q}, E[p^\infty])_{\mathrm{div}}$ is as the same as the Kummer image of $E(\mathbb{Q}) \otimes \mathbb{Q}_p/\mathbb{Z}_p$ in $\mathrm{Sel}(\mathbb{Q}, E[p^\infty])$.
Write $\partial = \partial^{(\mathrm{ord}(\ks^{\mathrm{Hg}}))}(\ks^{\mathrm{Hg}})$ for convenience.
For every integer $k \geq 1$ with $p^{k+\partial-1} \mathrm{Sel}(K, E[p^\infty])_{/\mathrm{div}} = 0$ and every $n_k \in \mathcal{N}_{k+\partial}$ with $\nu(n_k) = r^+_p - 1$,
we have
\begin{align*}
\kappa^{\mathrm{Hg}}_{n_k} & \in p^{\partial} \mathrm{Sel}(\mathbb{Q}, E[p^{k+\partial}]) \\
& \subseteq \mathrm{Sel}(\mathbb{Q}, E[p^{\infty}])_{\mathrm{div}} \\
& = \textrm{the Kummer image of } E(\mathbb{Q}) \otimes \mathbb{Q}_p/\mathbb{Z}_p
\end{align*}
where the first inclusion follows from Corollary \ref{cor:unramified-everywhere} and the second inclusion follows from the annihilation $p^{\partial} \mathrm{Sel}(K, E[p^\infty])_{/\mathrm{div}} = 0$.
By analyzing the diagram (\ref{eqn:kummer-map-Kn}) and the shape of the representing cocycle 
$\sigma \mapsto \dfrac{1}{p^{k+\partial}} \cdot \left( P^{\sigma}_{n_k} - P_{n_k}\right)$ of the image of $\kappa^{\mathrm{Hg}}_{n_k}$ in $\mathrm{H}^1(K[n_k]/K, E(K[n_k]))$ described in \S\ref{subsubsec:derived-heegner-points}, 
the triviality of the image of $\kappa^{\mathrm{Hg}}_{n_k}$ in $\sha(E/\mathbb{Q})[p^\infty]$ is equivalent to the inclusion
 $$P_{n_k} \in p^{ \partial }\left( E(\mathbb{Q}) +p^kE(K[n_k]) \right).$$ 
 See also \cite[Prop. 1.2.3(2)]{ciperiani-wiles}.
\subsubsection{$(2) \Rightarrow (3)$}
This is trivial.
\subsubsection{$(3) \Rightarrow (1)$}
Following Remark \ref{rem:tate-shafarevich-L}, we have
$$\mathrm{Sel}(\mathbb{Q}, E[p^\infty])_{\mathrm{div}} = \bigcup_{k\geq 1} \mathrm{Hg}_k $$
For every $\kappa^{\mathrm{Hg}}_{n'_i}$ in $\mathrm{Hg}_k$ (for some $k \geq 1$),
we assume that the corresponding derived Heegner point $P_{n'_i}$ satisfies
 $$P_{n'_i} \in p^{ \partial }\left( E(\mathbb{Q}) +p^kE(K[n'_i]) \right),$$
 and it is equivalent to that the image of  $\kappa^{\mathrm{Hg}}_{n'_i}$  in $\sha(E/\mathbb{Q})[p^\infty]$  is trivial. Thus, it  lies in the Kummer image of $E(\mathbb{Q}) \otimes \mathbb{Q}_p/\mathbb{Z}_p$.
This implies that 
$$\mathrm{Sel}(\mathbb{Q}, E[p^\infty])_{\mathrm{div}} \subseteq \textrm{the Kummer image of } E(\mathbb{Q}) \otimes \mathbb{Q}_p/\mathbb{Z}_p.$$
The conclusion follows.
\subsubsection{A natural question}
Theorem \ref{thm:main-finiteness} tells us that the finiteness of $\sha(E/\mathbb{Q})[p^\infty]$ controls the location of derived Heegner points very strongly and vice versa. The following (meta-)question is natural and fundamental for the finiteness of $\sha(E/\mathbb{Q})[p^\infty]$.
\begin{ques}
Which principle for Heegner points would imply that
$$ P_{n_k} \in p^{ \partial^{(\mathrm{ord}(\ks^{\mathrm{Hg}})) } (\ks^{\mathrm{Hg}}) }\left( E(\mathbb{Q}) +p^kE(K[n_k]) \right)$$
for every integer $k \geq 1$ with $p^{k+\partial-1} \mathrm{Sel}(K, E[p^\infty])_{/\mathrm{div}} = 0$ and every $n_k \in \mathcal{N}_{k+\partial}$ with $\nu(n_k) = r^+_p - 1$?
For example, how to obtain $\mathrm{cork}_{\mathbb{Z}_p} \mathrm{H}^1(L/K, E(L))[p^\infty] = 0$ for $L$ in Remark \ref{rem:tate-shafarevich-L}?
\end{ques}


\subsection{The local-global principle on the linear dependence in Mordell--Weil groups} \label{subsec:linear-dependence}

\subsubsection{The local counterpart of Theorem \ref{thm:main-finiteness}}
We expand Theorem \ref{thm:main-finiteness} as follows.
\begin{thm} \label{thm:main-finiteness-local-criteria}
We keep all the assumptions in Theorem \ref{thm:main-finiteness}.
Fix an integer $k \geq 1$ with $p^{k+\partial-1} \mathrm{Sel}(K, E[p^\infty])_{/\mathrm{div}} = 0$.
Then any equivalent statement (1), (2), or (3) in Theorem \ref{thm:main-finiteness} is also equivalent to any of the following statements. 
\begin{enumerate}
\item[(4)]
 For every $1 \leq i \leq r^+_p$, the derived Heegner point $P_{n_{k, i}} \in E(K[n_{k, i}])$ in (\ref{eqn:generators-derived-heegner-points}) satisfies
$$\mathrm{red}_{v_{k,i}} \left( P_{n_{k, i}} \right)  \in \mathrm{red}_{v_{k,i}} \left(  p^{ \partial }\left( E(\mathbb{Q}) +p^kE(K[n_{k, i}]) \right) \right) $$
for all but finitely many primes $ v_{k,i}$ of $K[n_{k, i}]$
 where
$\mathrm{red}_{v_{k,i}} : E(K[n_{k,i}]) \to \widetilde{E}(\mathbb{F}_{v_{k,i}})$ is the reduction map and $\mathbb{F}_{v_{k,i}}$ is the residue field of $K[n_{k,i}]$ at $v_{k,i}$.
\item[(5)]
 For every $1 \leq i \leq r^+_p$, the derived Heegner point $P_{n_{k, i}} \in E(K[n_{k, i}])$ in (\ref{eqn:generators-derived-heegner-points}) satisfies
$$\mathrm{red}_{v_{k,i}} \left( P_{n_{k, i}} \right)  \in \mathrm{red}_{v_{k,i}} \left(  p^{ \partial }\left( E(\mathbb{Q}) +p^kE(K[n_{k, i}]) \right) \right) $$
for every $ v_{k,i} \in S^{\mathrm{fin}}_i$, a finite subset of primes of $K[n_{k, i}]$ depending on 
$E$, $P_{n_{k, i}}$,  $p^{ \partial }\left( E(\mathbb{Q}) +p^kE(K[n_{k, i}]) \right)$, and a basis of 
$E(K[n_{k, i}])$.
\end{enumerate}
\end{thm}
\begin{proof}
The equivalence follows essentially from the local-global principle on the linear dependence in Mordell--Weil groups (Theorem \ref{thm:local-global}). In particular, we refer \cite[Thm. 5.1]{gajda-gornisiewicz-linear-dependence} for (4) and \cite[Thm. 6.4]{banaszak-krason-mordell-weil} for (5).
See \cite[pp. 334--335]{banaszak-krason-mordell-weil} for the explicit construction of the \emph{finite} set $S^{\mathrm{fin}}_i$.

\end{proof}
Since the inclusion $\mathrm{red}_{v_{k,i}} \left( P_{n_{k, i}} \right) \in  p^{ \partial }\left( \widetilde{E}(\mathbb{F}_v) +p^k\widetilde{E}(\mathbb{F}_{v_{k,i}}) \right) $ follows immediately from the definition of Selmer groups
and 
$$\mathrm{red}_{v_{k,i}} \left(  p^{ \partial }\left( E(\mathbb{Q}) +p^kE(K[n_{k, i}]) \right) \right) \subseteq    p^{ \partial }\left( \widetilde{E}(\mathbb{F}_v) +p^k\widetilde{E}(\mathbb{F}_{v_{k,i}}) \right),$$
Theorem \ref{thm:main-finiteness-local-criteria} proposes more precise locations of the derived Heegner points in (\ref{eqn:generators-derived-heegner-points}) in the reduction of elliptic curves at (finitely many) various primes.
Here,  $\mathbb{F}_v$ is the residue field of $\mathbb{Q}$ at $v$ lying below $v_{k,i}$.
In this sense, the fundamental obstruction for the finiteness of $\sha(E/\mathbb{Q})[p^\infty]$ lies in understanding 
$\mathrm{red}_{v_{k,i}} \left(  p^{ \partial }\left( E(\mathbb{Q}) +p^kE(K[n_{k, i}]) \right) \right)$
for (finitely many) primes $v_{k,i}$ of $K[n_{k,i}]$ and every $1 \leq i \leq r^+_p$.
It is natural to ask the following question.
\begin{ques}
Is there an efficient algorithm to compute
$$\mathrm{red}_{v_{k,i}} \left(  p^{ \partial }\left( E(\mathbb{Q}) +p^kE(K[n_{k, i}]) \right) \right)$$
\emph{without} computing global generators of $p^{ \partial }\left( E(\mathbb{Q}) +p^kE(K[n_{k, i}]) \right)$?
\end{ques}
If the answer of the above question is affirmative, then the verification of the finiteness of $\sha(E/\mathbb{Q})[p^\infty]$ would reduce to finitely many steps of algorithmically easy computations.

\subsubsection{The local-global principle}
We recall the local-global principle on the linear dependence in Mordell--Weil groups for elliptic curves \cite{weston-kummer, gajda-gornisiewicz-linear-dependence, banaszak-krason-mordell-weil, jossen-linear-reduction}.
\begin{thm} \label{thm:local-global}
Let $E$ be an elliptic curve over $\mathbb{Q}$, $F$ be a number field, $P$ be a non-torsion point of $E(F)$, and $\Lambda$ be a subgroup of $E(F)$. 
The following statements are equivalent.
\begin{enumerate}
\item $P \in \Lambda$.
\item $\mathrm{red}_v(P) \in \mathrm{red}_v(\Lambda)$ for every prime $v$ of $F$.
\item $\mathrm{red}_v(P) \in \mathrm{red}_v(\Lambda)$ for all but finitely many primes $v$ of $F$.
\item $\mathrm{red}_v(P) \in \mathrm{red}_v(\Lambda)$ for every primes $v$ of a subset of the primes of $F$ of natural density 1.
\item $\mathrm{red}_v(P) \in \mathrm{red}_v(\Lambda)$ for every primes $v$ of a finite subset of the primes of $F$ depending on $E$, $P$, $\Lambda$, and a basis of $E(F)$.
\end{enumerate}
\end{thm}
\begin{proof}
The direction (1) $\Rightarrow$ (2)  $\Rightarrow$ [any of (3), (4), and (5)] is obvious.
The implication (3) $\Rightarrow$ (1) follows from \cite{gajda-gornisiewicz-linear-dependence}. See also \cite{weston-kummer}.
The implication (4) $\Rightarrow$ (1) follows from \cite{jossen-linear-reduction}.
The implication (5) $\Rightarrow$ (1) follows from \cite{banaszak-krason-mordell-weil}.
\end{proof}
In order to make the local criterion effective, we need to be able to compute $\mathrm{red}_v(\Lambda)$ \emph{without} computing a set of generators of $\Lambda$ or $E(F)$.

\section*{Acknowledgement}
This article has grown out as a side project from a number of discussions with Minhyong Kim for the last three years.
Our (still ongoing) discussion is originally motivated by investigating some refined aspects of Remark \ref{rem:tate-shafarevich-L}, which is not observed in the general theory of Euler systems.
Although he did not sign as a co-author of this article in the end, I regard it as a joint work with him.
I would like to thank Grzegorz Banaszak, Ashay Burungale, Francesc Castella, Henri Darmon, Keunyoung Jeong, Aditya Karnataki, Shinichi Kobayashi, Masato Kurihara, Kentaro Nakamura, Gyujin Oh, and Robert Pollack for their encouragement and helpful discussions.

\bibliographystyle{amsalpha}
\bibliography{library}

\end{document}